\documentclass[11pt]{article}

\usepackage[english]{babel}
\usepackage{authblk}

\title{Cluster parking functions}

\author[1]{Theo Douvropoulos}
\author[2]{Matthieu Josuat-Vergès}

\affil[1]{Brandeis University}
\affil[2]{IRIF, CNRS, Université Paris-Cité}

\usepackage{stmaryrd}
\usepackage{amsmath,amssymb,amsthm,mathtools}
\usepackage[linktocpage,unicode]{hyperref}
\usepackage{bm}

\usepackage{tikz}

\usepackage{a4wide}

\usepackage{caption}
\hypersetup{
    pdftoolbar=true,        
    pdfmenubar=true,        
    pdffitwindow=false,     
    pdfstartview={FitH},    
    pdftitle={},
    pdfauthor={},
    pdfsubject={},
    pdfkeywords={},
    pdfnewwindow=true,      
    colorlinks=true,       
    linkcolor=red,          
    citecolor=blue,        
    filecolor=magenta,      
    urlcolor=cyan,           
    unicode=true          
}

\newtheorem{theo}{Theorem}[section]

\newtheorem{lemm}[theo]{Lemma}
\newtheorem{prop}[theo]{Proposition}
\newtheorem{conj}[theo]{Conjecture}

\theoremstyle{definition}
\newtheorem{defi}[theo]{Definition}
\newtheorem{rema}[theo]{Remark}

\newtheorem{open}[theo]{Open problem}

\DeclareMathOperator{\NC}{NC}
\DeclareMathOperator{\LC}{LC}
\DeclareMathOperator{\PF}{PF}
\DeclareMathOperator{\CPF}{CPF}
\DeclareMathOperator{\bNC}{\overline{NC}}
\DeclareMathOperator{\bPF}{\overline{PF}}

\DeclareMathOperator{\Cat}{Cat}
\DeclareMathOperator{\Inv}{Inv}

\DeclareMathOperator{\Red}{Red}
\DeclareMathOperator{\Span}{Span}

\DeclareMathOperator{\rk}{rk}

\newcommand{\bchi}{\boldsymbol{\chi}}
\newcommand{\bzeta}{\boldsymbol{\zeta}}
\newcommand{\park}{\mathbf{Park}}
\newcommand{\sign}{\mathbf{Sign}}
\newcommand{\ch}{\mathbf{Ch}}
\newcommand{\bC}{\mathbf{C}}
\newcommand{\bH}{\mathbf{H}}

\begin{document}

\maketitle

\begin{abstract}
    The cluster complex on one hand, parking functions on the other hand, are two combinatorial (po)sets that can be associated to a finite real reflection group.  Cluster parking functions are obtained by taking an appropriate fiber product (over noncrossing partitions).  There is a natural structure of simplicial complex on these objects, and our main goal is to show that it has the homotopy type of a (pure) wedge of spheres.  The unique nonzero homology group (as a representation of the underlying reflection group) is a sign-twisted parking representation, which is the same as Gordon's quotient of diagonal coinvariants.  
    Along the way, we prove some properties of the poset of parking functions.  We also provide a long list of open problems. 
\end{abstract}



\section{Introduction}

\subsection{Parking functions} Before going into the general framework of finite real reflection groups, let us illustrate our main points in the case of the symmetric group.  A {\it parking function} can be represented as a Dyck path, together with some labels, for example
\[
    \begin{tikzpicture}[scale=0.4]
        \draw[lightgray] (0,0) grid (8,8);
        \draw[very thick] (0,0) -- (0,3) -- (1,3) -- (1,4) -- (3,4) -- (3,8) -- (8,8);
        \draw[dashed] (0,0) -- (8,8);
        \node at (0.5,0.5) {$2$};
        \node at (0.5,1.5) {$3$};
        \node at (0.5,2.5) {$7$};
        \node at (1.5,3.5) {$4$};
        \node at (3.5,4.5) {$1$};
        \node at (3.5,5.5) {$5$};
        \node at (3.5,6.5) {$6$};
        \node at (3.5,7.5) {$8$};
    \end{tikzpicture}
\]
is a parking function of size $n=8$.  The rule is that each $i\in\{1, \dots, n\}$ appears exactly once as a label of some vertical step, and labels are increasing along sequences of consecutive vertical steps.  Equivalently, there is a set partition $\pi$ of $\{1,\dots,n\}$ such that each sequence of $k$ vertical steps in the Dyck path is associated to a $k$-element block of $\pi$.  

Generalizing the $(n\times n)$-square in the example above, one can consider an $(n\times mn)$-rectangle (with $n$ rows and $mn$ columns, where $m\geq 1$) and its main diagonal.  The paths above the diagonal are {\it $m$-Dyck paths}, counted by the {\it Fuß-Catalan numbers}: \begin{align} \label{eq:deffusscat}
    C_n^{(m)} 
        :=
    \frac{1}{(m+1)n+1}\binom{(m+1)n+1}{n}.
\end{align}
The associated labeled paths give the set $\PF_n^{(m)}$ of {\it $m$-parking functions}.  For example, an element of $\PF^{(3)}_5$ is
\[
    \begin{tikzpicture}[scale=0.4]
        \draw[lightgray] (0,0) grid (15,5);
        \draw[very thick] (0,0) -- (0,2) -- (3,2) -- (3,3) -- (8,3) -- (8,5) -- (15,5);
        \draw[dashed] (0,0) -- (15,5);
        \node at (0.5,0.5) {$2$};
        \node at (0.5,1.5) {$4$};
        \node at (3.5,2.5) {$3$};
        \node at (8.5,3.5) {$1$};
        \node at (8.5,4.5) {$5$};
    \end{tikzpicture}\;.
\]
Moreover, let $\PF_n'^{(m)}$ denote the subset of {\it prime} $m$-parking functions, {\it i.e.}, those such that the underlying path is a {\it prime} $m$-Dyck path (which means that it doesn't touch the diagonal except at the two opposite corners of the rectangle).  See Armstrong, Loehr and Warrington~\cite{armstrongloehrwarrington} (they define rational parking functions indexed by pairs of coprime integers $(a,b)$, and the two sets $\PF_n^{(m)}$ and $\PF_n'^{(m)}$ correspond to $(n,mn+1)$ and $(n,mn-1)$, respectively).  

For each $m$-Dyck path, we associate an integer partition $\lambda$ (its {\it type}) by the rule that $k$ vertical steps with the same abscissa correspond to a part $k$ in $\lambda$.  For example, the underlying paths in the examples above  respectively give $(4,3,1)$ and $(2,2,1)$.  The Frobenius characteristic of the symmetric group $\mathfrak{S}_n$ acting on $\PF^{(m)}$ and $\PF'^{(m)}$ by relabelling can be given explicitly~\cite{armstrongloehrwarrington}:
\begin{align} \label{eq:chpf}
    \operatorname{ch}(\PF_n^{(m)})
        =
    \sum_{\lambda \vdash n}  K_{\lambda,m} h_\lambda
        =
    \sum_{\lambda \vdash n} (mn+1)^{\ell(\lambda)} \frac{p_\lambda}{z_\lambda}, \\
    \label{eq:chpf'}
    \operatorname{ch}(\PF_n'^{(m)})
        =
    \sum_{\lambda \vdash n}  K'_{\lambda,m} h_\lambda
        =
    \sum_{\lambda \vdash n} (mn-1)^{\ell(\lambda)} \frac{p_\lambda}{z_\lambda},
\end{align}
where:
\begin{itemize}
    \item $K_{\lambda,m}$ (respectively, $K'_{\lambda,m}$) is the number of $m$-Dyck paths (respectively, prime $m$-Dyck paths) of type $\lambda$,
    \item $\ell(\lambda)$ is the length of $\lambda$, $h_\lambda$ are the homogeneous symmetric functions, $p_\lambda$ are the powersum symmetric functions, $z_\lambda$ is a combinatorial factor (see Macdonald~\cite{macdonald}).
\end{itemize}
Moreover, there are explicit formulas for $K_{\lambda,m}$ and $K'_{\lambda,m}$:
\begin{align} \label{eq:formulaK}
      K_{m,\lambda} = \frac{\prod_{i=1}^{\ell(\lambda)-1}(mn+1-i) }{ \prod_{i\geq 1} \mu_i(\lambda)! },
        \qquad
      K'_{m,\lambda} = \frac{\prod_{i=1}^{\ell(\lambda)-1}(mn-1-i) }{ \prod_{i\geq 1} \mu_i(\lambda)! },
\end{align}
where $\mu_i(\lambda)$ is the multiplicity of $i$ in $\lambda$.

Loosely speaking, our cluster parking functions are similar to parking functions in the sense that we put labels on faces of the {\it cluster complex}, rather than Dyck paths. 

\subsection{The cluster complex} Consider the simplicial complex $\Delta_n$ of {\it dissections} of a convex $(n+2)$-gon: its vertices are the $\frac{(n-1)(n+2)}{2}$ diagonals, its faces are sets of pairwise noncrossing diagonals, and its facets are triangulations.  This is the combinatorial description of the cluster complex of type $A_{n-1}$, see~\cite{fominzelevinsky}.  For example, a face and a facet of $\Delta_6$ are:
\[
    \begin{tikzpicture}
        \draw[very thick] (0.38,-0.92) -- (-0.38,-0.92) -- (-0.92,-0.38) -- (-0.92,0.38) -- (-0.38,0.92) -- (0.38,0.92) -- (0.92,0.38) -- (0.92,-0.38) -- (0.38,-0.92);
        \draw[very thick] (-0.38,-0.92) -- (0.38,0.92);
        \draw[very thick] (-0.38,-0.92) -- (0.92,0.38);
    \end{tikzpicture}
    ,\qquad 
    \begin{tikzpicture}
        \draw[very thick] (0.38,-0.92) -- (-0.38,-0.92) -- (-0.92,-0.38) -- (-0.92,0.38) -- (-0.38,0.92) -- (0.38,0.92) -- (0.92,0.38) -- (0.92,-0.38) -- (0.38,-0.92);
        \draw[very thick] (-0.92,-0.38) -- (-0.38,0.92);
        \draw[very thick] (-0.92,-0.38) -- (0.38,0.92);
        \draw[very thick] (-0.38,-0.92) -- (0.38,0.92);
        \draw[very thick] (-0.38,-0.92) -- (0.92,0.38);
        \draw[very thick] (-0.38,-0.92) -- (0.92,-0.38);
    \end{tikzpicture}.
\]
The {\it Fuß analog} can be described similarly, following Fomin and Reading~\cite{fominreading}.  Let $m\geq 1$ be an integer (a {\it Fuß parameter}).  The {\it generalized cluster complex} $\Delta^{(m)}_n$ (of type $A_{n-1}$) is the complex of dissections of an $(mn+2)$-gon, where each inner polygon in the dissection is a $(mk+2)$-gon for some $k \geq 1$.  The number of facets in $\Delta^{(m)}_n$ is the Fuß-Catalan number $C_n^{(m)}$ as given in~\eqref{eq:deffusscat}. 

To each face $f \in \Delta^{(m)}_n$, we associate an integer partition $\lambda = (\lambda_1 ,\lambda_2,\dots )$ of $n$ (called its {\it type}, as with parking functions and Dyck paths).  The rule is that each inner $(mk+2)$-gon gives a part of $\lambda$ equal to $k$.  In the example above, we get $\lambda=(3,2,1)$ for the face with two diagonals, and $\lambda=(1,1,1,1,1,1)$ for the facet (or any other facet).  This integer partition has a geometric interpretation: the link of the face $f$ in $\Delta^{(m)}_n$ is isomorphic to the product $\Delta^{(m)}_{\lambda_1} \times \Delta^{(m)}_{\lambda_2} \times \cdots$ (this is especially relevant when $m=1$, where the complex $\Delta_n$ is the dual of a simple polytope called the {\it associahedron}:  for a face $f\in \Delta_n$ of type $\lambda$, the corresponding face in the associahedron is combinatorially isomorphic to a product of smaller associahedra of dimensions $\lambda_1-1,\lambda_2-1, \dots$).  

There is a natural {\it positive} subcomplex $\Delta^{(m),+}_n \subset \Delta^{(m)}_n$.  The number of faces $f \in \Delta^{(m)}_n$ (respectively, $f \in \Delta^{(m),+}_n$) of a given type $\lambda$ is denoted by $L_{m,\lambda}$ (respectively, $L^+_{m,\lambda}$).  One result from our previous work is the type-refined enumeration of faces of the generalized cluster complex, and it can be obtained from the relations~\cite[Theorem~7.1]{douvropoulosjosuatverges}:
\begin{equation} \label{eq!llambda}
    L_{m,\lambda} = (-1)^{\ell(\lambda)} K'_{-m,\lambda},
        \qquad
    L^+_{m,\lambda} = (-1)^{\ell(\lambda)} K_{-m,\lambda}.
\end{equation}
In particular, explicit formulas can be deduced using those in~\eqref{eq:formulaK}.  From~\eqref{eq:chpf} and~\eqref{eq:chpf'} together with the previous equation, we get:
\begin{align}    \label{altern_sum}
    \sum_{\lambda \vdash n}  (-1) ^{\ell(\lambda)} L_{m,\lambda} h_\lambda
        =
    \sum_{\lambda \vdash n }(-mn-1)^{\ell(\lambda)} \frac{p_\lambda}{z_\lambda},\\
    \label{altern_sum2}
    \sum_{\lambda \vdash n}  (-1) ^{\ell(\lambda)} L^+_{m,\lambda} h_\lambda
        =
    \sum_{\lambda \vdash n }(-mn+1)^{\ell(\lambda)} \frac{p_\lambda}{z_\lambda}.
\end{align}
Our cluster parking functions will give a representation theoretic interpretation of~\eqref{altern_sum} and~\eqref{altern_sum2}. 

\subsection{Cluster parking functions}
At the origin of this work, is the idea that the alternating sum in the left-hand side of~\eqref{altern_sum} (respectively,~\eqref{altern_sum2}) can be interpreted as the reduced equivariant Euler characteristic\footnote{The term  ``equivariant Euler characteristic'' is far from standard, and indeed has sometimes different meanings in the literature.  See the next section for its meaning here.} of a simplicial complex $\CPF^{(m)}_n$ (respectively, $\CPF^{(m),+}_n$).  An element of $\CPF^{(m)}_n$, called a {\it cluster parking function}, is a face $f\in \Delta^{(m)}_n$ together with labels such that each inner $(mk+2)$-gon has $k$ (unordered) labels and each $i\in \{1,\dots,n \}$ appears exactly once as a label.  For example, a 2-dimensional face of $\CPF_6$ (the absence of the Fuß parameter $m$ means that $m=1$) is:
\[
    \begin{tikzpicture}
        \draw[very thick] (0.38,-0.92) -- (-0.38,-0.92) -- (-0.92,-0.38) -- (-0.92,0.38) -- (-0.38,0.92) -- (0.38,0.92) -- (0.92,0.38) -- (0.92,-0.38) -- (0.38,-0.92);
        \draw[very thick] (-0.38,-0.92) -- (0.38,0.92);
        \draw[very thick] (-0.38,-0.92) -- (0.92,0.38);
        \node at (-0.4,0.2) {\small 235};
        \node at (0.45,0.35)  {\small 4};
        \node at (0.5,-0.4)  {\small 16};
    \end{tikzpicture}.
\]
A smaller face is obtained by removing some diagonals (as in $\Delta^{(m)}_n$, the complex of dissections), and when inner faces merge we accordingly merge their label sets.  For example, the smaller faces obtained from the face in the example above are:
\[
    \begin{tikzpicture}
        \draw[very thick] (0.38,-0.92) -- (-0.38,-0.92) -- (-0.92,-0.38) -- (-0.92,0.38) -- (-0.38,0.92) -- (0.38,0.92) -- (0.92,0.38) -- (0.92,-0.38) -- (0.38,-0.92);
        \draw[very thick] (-0.38,-0.92) -- (0.38,0.92);
        \node at (-0.4,0.2) {\small 235};
        \node at (0.45,-0.2)  {\small 146};
    \end{tikzpicture},
    \qquad
        \begin{tikzpicture}
        \draw[very thick] (0.38,-0.92) -- (-0.38,-0.92) -- (-0.92,-0.38) -- (-0.92,0.38) -- (-0.38,0.92) -- (0.38,0.92) -- (0.92,0.38) -- (0.92,-0.38) -- (0.38,-0.92);
        \draw[very thick] (-0.38,-0.92) -- (0.92,0.38);
        \node at (-0.2,0.2) {\small 2345};
        \node at (0.5,-0.4)  {\small 16};
    \end{tikzpicture},
    \qquad
            \begin{tikzpicture}
        \draw[very thick] (0.38,-0.92) -- (-0.38,-0.92) -- (-0.92,-0.38) -- (-0.92,0.38) -- (-0.38,0.92) -- (0.38,0.92) -- (0.92,0.38) -- (0.92,-0.38) -- (0.38,-0.92);
        \node at (0,0) {\small 123456};
    \end{tikzpicture},
\]
(the last element being the empty face).  Moreover, the group $\mathfrak{S}_n$ acts on $\CPF_n^{(m)}$ and  $\CPF_n^{(m),+}$ by relabeling just as in the case of parking functions.  

Now, our main result says that $\CPF^{(m)}$ and $\CPF^{(m),+}$ have the homotopy type of wedges of $(n-2)$-dimensional spheres.  It follows that the characters of $\mathfrak{S}_n$ acting on their degree $n-2$ homology are the symmetric functions in~\eqref{altern_sum} and~\eqref{altern_sum2}, up to a factor $(-1)^n$.  Let us complete the picture by stating the main theorem in the general case of finite real reflection groups. 





\begin{theo}[{Theorem \ref{theo:topcpf} below}]
Let $W$ be a finite irreducible real reflection group of rank $n$, and let $h$ be its Coxeter number.  The simplicial complexes $\CPF^{(m)}(W)$ and $\CPF^{(m),+}(W)$ (defined in Section~\ref{sec:cpf}) have the homotopy type of a wedge of $(n-1)$-dimensional spheres, consisting of $(mh+1)^n$ spheres for the former and $(mh-1)^n$ spheres for the latter.  Moreover, the character of $W$ acting on the nonzero homology group $\tilde H_{n-1}(\CPF^{(m)}(W))$, respectively $\tilde H_{n-1}(\CPF^{(m),+}(W))$, is 
\begin{equation}  \label{eq:charhcpf1}
    w \mapsto (-1)^{\ell(w)} (mh+1)^{n-\ell(w)},
\end{equation}
respectively
\begin{equation}  \label{eq:charhcpf2}
    w \mapsto (-1)^{\ell(w)} (mh-1)^{n-\ell(w)}
\end{equation}
({\it i.e.}, these are parking characters twisted by the sign character).
\end{theo}  

The complex $\CPF_n^+$ (in the case of the symmetric group) was also defined in~\cite{delcroixogerjosuatvergesrandazzo}, where we study the topology of a poset structure on parking functions (denoted by $\PF_n$).  The hope was to use cluster parking functions in order to get a nice derivation of the topology of $\PF_n$.  This turned out to be the other way around: an important step of the present work, which is also interesting on its own, is to obtain the topology of $\PF(W)$ as a wedge of spheres (see Theorem~\ref{theo:pftop}).

A few remarks are in order about the general case of finite real reflection groups.   All the character calculations are amenable because they take place in the {\it parabolic Burnside ring} (which is very well-understood in a uniform way, unlike the full character ring).  Various relevant character calculations have been presented in our previous work~\cite{douvropoulosjosuatverges}, to which we refer when needed.  Our combinatorial construction relies on {\it noncrossing parking functions}, introduced by Armstrong, Reiner and Rhoades~\cite{armstrongreinerrhoades} in the context of their parking space theory.  Here (Section~\ref{sec:toppf}), we examine the topology of an associated partial order, generalizing the type $A$ results in~\cite{delcroixogerjosuatvergesrandazzo}.  

To finish this introduction, let us mention that the parking character twisted by sign (the case $m=1$ in~\eqref{eq:charhcpf1}) is also the (ungraded) character of Gordon's ring $R_W$ defined in terms of diagonal coinvariants~\cite{gordon}.  So $\tilde H_{n-1}(\CPF)$ is isomorphic to this ring as a representation of $W$.  It is unclear whether this is just a mere coincidence or not. 

Let us now outline the contents of this article.  After some preliminaries in Section~\ref{sec:defi}, we give first properties of cluster parking functions in Section~\ref{sec:cpf}.  The proof of the main result has three steps:
\begin{itemize}
    \item In Section~\ref{sec:topnc}, we give a preliminary result about the topology of some order ideals in the poset of noncrossing partitions.
    \item In Section~\ref{sec:toppf}, we describe the topology of the parking function poset $\PF$ (using the result in Section~\ref{sec:topnc}).  
    \item In Section~\ref{sec:topcpf}, we use a poset map $\CPF \to \PF$ and Quillen's theory of Cohen-Macaulay posets to arrive at our main result, Theorem~\ref{theo:topcpf}.
\end{itemize}
In Section~\ref{ssec:h_vecs}, we give positive formulas for the $f$- and $h$-vectors of the complexes $\CPF$ involving numerical invariants of the corresponding reflection group.  In Section~\ref{sec:enum}, we introduce new combinatorial objects that are equinumerous with parking functions as a consequence of the results in Section~\ref{sec:toppf}.  Several open problems are given in Section~\ref{sec:open}.

\section{Definitions and background}
\label{sec:defi}

\subsection{Poset topology}

We refer to Wachs~\cite{wachs} or Kozlov~\cite{kozlov} for this topic.

Let $\Gamma$ be any finite simplicial complex, that we identify with its face poset (where the order relation is inclusion).  We say $\Gamma$ is {\it pure} if all its facets (maximal faces) have the same dimension.  We denote by $|\Gamma|$ the {\it geometric realization} of $\Gamma$.  We denote by $\Gamma_d$ its set of $d$-dimensional faces, for $d\geq -1$.  If we have an {\it orientation} of $\Gamma$ (each face $f$ is endowed with a total order $\omega_f$, and if $f'\subset f$ then $\omega_{f'}$ is the total order induced by $\omega_f$), we can define the {\it reduced simplicial homology groups} over $\mathbb{Z}$, denoted by $\tilde H_i(\Gamma)$ for $i\geq -1$.  The {\it reduced Euler characteristic} $\tilde\chi$ of $\Gamma$ is:
\[
    \tilde \chi(\Gamma) 
    :=
    \sum_{i \geq -1} (-1)^i \# \Gamma_i
    =
    \sum_{i \geq -1} (-1)^i \dim_{\mathbb{Q}} (\mathbb{Q}\otimes \tilde H_i(\Gamma) ).
\]

Let $G$ be finite group acting on $\Gamma$, so that inclusion and dimension of faces are preserved.   Let $\Gamma^g := \{ f\in \Gamma \;:\; g\cdot f = f\}$ be the fixed-point set of $g\in G$.  We assume that the group action is {\it admissible}, which means that for all $g\in G$, $\Gamma^g$ is a subcomplex of $\Gamma$ (an order ideal in the face poset).  Let $\bC_i(\Gamma)$ denote the character of $G$ acting on $\Gamma_i$ (the $i$th chain space).  There is also an induced action of $G$ on the homology group $\tilde H_i(\Gamma)$.  We denote by $\tilde\bH_i(\Gamma)$ its character.  The (reduced) {\it equivariant Euler characteristic} is the (virtual) character
\[
    \tilde \bchi(\Gamma) 
    := 
    \sum_{i\geq -1} (-1)^i \bC_i(\Gamma)
    =
    \sum_{i\geq -1} (-1)^i \tilde\bH_i(\Gamma),
\]
where the last equality comes from the {\it Hopf trace formula}.   The (reduced) {\it Lefschetz number} $\tilde \Lambda(\Gamma, g)$ is defined as $\tilde \chi(\Gamma^g)$.  Baclawski and Björner~\cite{baclawskibjorner} have showed that $\tilde \Lambda(\Gamma,g)$ is the evaluation of $\tilde \bchi(\Gamma)$ at $g$.

The {\it order complex} $\Omega(P)$ of a finite poset $P$ is the simplicial complex having the elements of $P$ as vertices, and sets of pairwise comparable elements as faces.  A $d$-dimensional face $f\in \Omega(P)$ can be represented as a strict chain $x_0 < x_1 < \dots < x_{d}$ in $P$.  In particular, the order induced by $P$ gives a natural orientation of $\Omega(P)$, which we always use to compute reduced homology.  As a shortcut, we denote by $\tilde H_i(P) := \tilde H_i(\Omega(P))$ and similarly for $\tilde \chi$, {\it etc}.  The poset $P$ is {\it ranked} iff $\Omega(P)$ is pure.  We write $\rk(p)$ for the rank of $p\in P$, omitting $P$ in the notation.  If the poset $P$ is the poset of nonempty faces of a simplicial complex $\Gamma$, the order complex $\Omega(P)$ is called the {\it barycentric subdivision} of $\Gamma$.  A simplicial complex and its barycentric subdivision have homeomorphic geometric realizations. 

The poset $P$ is called a $G$-poset if there is an action of $G$ preserving the order relation.  There is an induced action of $G$ on $\Omega(P)$, and it is admissible.  
The {\it $\zeta$-character} $\bzeta(P,m)$ of a $G$-poset $P$ is a polynomial in $m$ with values in the character ring of $G$ such that, for each integer $m\geq 1$, $\bzeta(P,m)$ is the character of $G$ acting on multichains $p_1 \leq \dots \leq p_m$.  If $P$ has a minimum $0$ but no maximum, we have the relation (see~\cite[Proposition~1.7]{delcroixoger}):
\begin{equation} \label{eq:relzetachi}
    \tilde \bchi(P\backslash\{0\})
        =
    - \bzeta(P,-1).
\end{equation}

For any poset $P$, denote by $\overline{P}$ the {\it proper part} of $P$, obtained by removing its minimum if it exists, and its maximum if it exists.  A finite poset $P$ is {\it bounded} if it has a minimum and a maximum.  If $P$ is not bounded, denote by $\hat P$ the smallest bounded poset containing $P$ (thus obtained by adding an extra top element when $P$ doesn't already have one, and an extra bottom element similarly).

Let $\mathbb{S}^d$ denote the standard $d$-dimensional sphere, and $\vee$ denote the wedge product (to be precise, the sphere should be considered as a {\it pointed} topological space to define wedge products).  The $k$-fold wedge product of $\mathbb{S}^d$ is thus denoted by $(\mathbb{S}^d)^{\vee k}$.  When $k=0$, this is the one-point topological space, denoted by $\bullet$.  Generally speaking, poset topology deals with the problem of finding the homotopy type of $\Omega(P)$ for a finite poset $P$.  Some of the most simple homotopy types are given by a wedge of spheres, and typical results have the form $\Omega(P) \simeq (\mathbb{S}^d)^{\vee k}$ for some nonnegative integers $d$ and $k$ ($\simeq$ stands for homotopy equivalence).

Let us also mention the following characterization of a wedge of spheres (see~\cite[Theorem~1.5.2]{wachs} or~\cite[Section~4.C]{hatcher}).

\begin{prop} \label{prop:characwedge}
    Let $\Gamma$ be a $d$-dimensional simplicial complex with $d\geq 2$ such that:
    \begin{itemize}
        \item $\Gamma$ is simply connected,
        \item $\tilde H_d(\Gamma)$ is isomorphic to $\mathbb{Z}^k$ for some integer $k\geq 0$,
        \item all other homology groups vanish.
    \end{itemize}
    Then $\Gamma \simeq (\mathbb{S}^d)^{\vee k}$.
\end{prop}

Note that the assumption $d\geq 2$ is not restrictive: for a $1$-dimensional simplicial complex $\Gamma$ with $\tilde H_0(\Gamma)=0$ ({\it i.e.}, $\Gamma$ is a connected graph), we always have $\Gamma \simeq (\mathbb{S}^1)^{\vee k}$ for some $k$.

\subsection{Finite real reflection groups and noncrossing partitions}

Through this work, $W$ is a finite real reflection group, $e$ its neutral element, $T\subset W$ the set of reflections, $S \subset T$ a set of simple reflections.  We don't assume $W$ to be irreducible.  For $I\subset S$, denote by $W_I$ the {\it standard parabolic subgroup} generated by $I$.  A {\it parabolic subgroup} $P\subset W$ is a subgroup conjugate to $W_I$ for some $I\subset S$.  

Let $V$ denote the geometric representation of $W$, and for $w\in W$ we denote by $ V^w \subset V$ its fixed-point set: $ V^w := \ker(w-e)$.  The {\it intersection lattice} $\mathcal{L}$ of $W$ is
\[
    \mathcal{L} := 
    \{
        V^w \;:\; w\in W
    \},
\]
where by convention the partial order is reverse inclusion. The map
\[
    F \mapsto \{ w \in W \;:\; F\subset V^w \}
\]
is a poset isomorphism from $\mathcal{L}$ to the poset of parabolic subgroups of $W$ (where the partial order is inclusion).

We refer to~\cite{armstrong} concerning noncrossing partitions and related material.  Let $c$ be a {\it standard Coxeter element}, {\it i.e.}, the product all simple reflections $s\in S$ in some (arbitrary) order.  Let $\ell$ denote the {\it reflection length} in $W$:
\[
    \ell(w) 
        := \min\{k\geq0 \;:\; \exists t_1,\dots,t_k \in T,\; w=t_1\cdots t_k\}.
\]
We also have:
\[
    \ell(w)
        =
    n - \dim(V^w).
\]
Note that $\ell(c)=n$.  The {\it absolute order} on $W$ is defined by 
\[
    v \geq w \Leftrightarrow \ell(v) + \ell(v^{-1}w) = \ell(w).
\]
Another characterization is the following: consider a minimal reflection factorization $w = t_1 \cdots t_k$ (which means that $k$ is minimal and each $t_i$ is a reflection), then $v\geq w$ iff a minimal reflection factorization of $v$ can be extracted as a subword of $t_1 \cdots t_k$.  We denote by $\NC$ the lattice of {\it noncrossing partitions}:
\[
    \NC := \big\{
        w \in W \;:\; w \geq c
    \big\},
\]
{\it i.e.}, this is the interval $[c,e]$ in the absolute order.  Moreover, denote by $\NC' \subset \NC$ the set of {\it prime noncrossing partitions}, {\it i.e.}, those $\pi\in\NC$ that don't belong to a proper standard parabolic subgroup of $W$.  The rank function on $\NC$ is given by
\[
    \rk(\pi) = n-\ell(\pi) = \dim(V^\pi).
\]

\begin{rema}
Our definition of $\NC$ and the absolute order is such that $c$ is the minimum, and $e$ is the maximum.  The dual convention (the other way around) is more common.  If we accept this unusual convention, some other common definitions will be respected and several poset maps will be order-preserving rather than order-reversing.
\end{rema}

For each $\pi \in \NC$, there is an associated {\it noncrossing parabolic subgroup} defined by
\[
    W_{\pi} := 
    \big\{
        w \in W \; : \; V^\pi \subset V^w
    \big\}.
\]
The map $\pi \mapsto W_{\pi}$ from $\NC$ to the lattice of parabolic subgroups is injective and order-reversing.  It can be used to describe to the join operation $\vee$ of the lattice $\NC$, as we have:
\begin{align}  \label{eq:joinnc}
    W_{\pi \vee \tau}
        =
    W_\pi \cap W_\tau.
\end{align}
This property follows from the proof of the lattice property by Brady and Watt~\cite[Section~7]{bradywatt}.


\subsection{Parking functions}
\label{sec:defpark}

We use here {\it noncrossing parking functions}, as introduced by Armstrong, Reiner and Rhoades in~\cite{armstrongreinerrhoades}.  The poset structure on type $A$ parking functions was introduced by Edelman~\cite{edelman} (on slightly different objects, called 2-noncrossing partitions), and further studied in~\cite{delcroixogerjosuatvergesrandazzo}.  
\begin{defi}
    The $W$-poset of {\it parking functions}, denoted by $\PF$, is defined as a set of cosets:
    \[
        \PF 
            :=
        \big\{
            w W_\pi \;:\; w\in W,\; \pi\in\NC
        \big\}.
    \]
    The partial order is reverse inclusion, and the action of $W$ is the natural one by left multiplication on right cosets.  A parking function $w W_\pi \in \PF$ is {\it prime} if $\pi\in\NC'$.  We denote by $\PF' \subset \PF$ the subposet of prime parking functions.  Note that the action clearly preserves inclusion, so that $\PF$ and $\PF'$ are indeed $W$-posets.
\end{defi}

Before stating elementary properties of the poset, let us quickly explain the link with ``classical'' parking functions.  As explained in~\cite[Section~2.3]{armstrongreinerrhoades}, noncrossing parking functions of type $A_{n-1}$ as defined above can be seen pictorially as (classical) noncrossing partitions with labels, for example:
\[
\begin{tikzpicture}[scale=0.25]
   \tikzstyle{ver} = [circle, draw, fill, inner sep=0.5mm]
   \tikzstyle{edg} = [line width=0.6mm]
   \tikzstyle{edg2} = [line width=0.6mm,dotted]
   \node      at (1,-1) {\small 2};
   \node      at (3,-1) {\small 3};
   \node      at (5,-1) {\small 4};
   \node      at (7,-1) {\small 7};
   \node      at (9,-1) {\small 1};
   \node      at (11,-1) {\small 5};
   \node      at (13,-1) {\small 8};
   \node      at (15,-1) {\small 10};
   \node      at (17,-1) {\small 9};
   \node      at (19,-1) {\small 6};
   \node[ver] at (1,0) {};
   \node[ver] at (3,0) {};
   \node[ver] at (5,0) {};
   \node[ver] at (7,0) {};
   \node[ver] at (9,0) {};
   \node[ver] at (11,0) {};
   \node[ver] at (13,0) {};
   \node[ver] at (15,0) {};
   \node[ver] at (17,0) {};
   \node[ver] at (19,0) {};
   \draw[edg] (1,0) to[bend left=60] (5,0);
   \draw[edg] (5,0) to[bend left=60] (7,0);
   \draw[edg] (9,0) to[bend left=60] (17,0);
   \draw[edg] (11,0) to[bend left=60] (13,0);
   \draw[edg] (13,0) to[bend left=60] (15,0);
 \end{tikzpicture}.
\]
The underlying noncrossing partitions is $\pi = \{\{1,3,4\},\{2\},\{5,9\},\{6,7,8\},\{10\}\}$ (this is obtained via the standard labelling with $1,2,\dots$ from left to right, unwritten here), and the noncrossing parabolic subgroup is $\mathfrak{S}_\pi = \mathfrak{S}_{\{1,3,4\}} \times \mathfrak{S}_{\{5,9\}} \times \mathfrak{S}_{\{6,7,8\}}$ (a product of symmetric groups).  The labels correspond to the permutation $\sigma = 2347158(10)96$, which is a minimal length coset representative in $\mathfrak{S}_{10} / \mathfrak{S}_\pi$ (meaning here that the labels are increasing in each block of $\pi$, from left to right).  The order relation and other combinatorial properties (in particular the link with a more classical definition of parking functions) on these type $A_{n-1}$ objects have been discussed in~\cite{delcroixogerjosuatvergesrandazzo}.

The {\it coset poset} of a group was introduced by Brown~\cite{brown}, who obtained some of its topological properties.  Up to some details (Brown only considers proper cosets, and the order is inclusion rather than reverse inclusion), $\PF$ is a subposet of this coset poset.  Note also that a subposet of $\PF$ is the {\it Coxeter complex}, which is the poset of cosets $w W_I$ where $W_I$ is a standard parabolic subgroup. 


It is clear that $\PF$ has a minimum, namely $W$ itself, and $|W|$ maximal elements, namely the singletons $\{w\}$ for $w \in W$.  (Representation-theoretically, we have the trivial representation in rank $0$ and the regular representation in rank $n$.)  The poset $\PF$ is ranked, and the rank function is such that
\[
    \rk(wW_\pi) = \rk(\pi)
\]
(see below for details).

\begin{lemm}
    Each element of $\PF$ can be uniquely written $w W_{\pi}$, where $\pi  \in \NC$ and $w$ is a minimal length coset representative in $W / W_{\pi}$.
\end{lemm}

\begin{proof}
    Uniqueness of $\pi$ follows from the fact that each orbit of cosets contains a unique subgroup, together with the injectivity of the map $\pi \mapsto W_{\pi}$ defined for $\pi\in \NC$.  We then use uniqueness of the minimal length coset representative.
\end{proof}

In what follows, by writing $w W_{\pi} \in \PF$ we always assume that $\pi \in \NC$ (and make explicit when we assume that $w$ is a minimal length coset representative).  The previous lemma makes clear that the orbits of $\PF$ are naturally indexed by $\NC$, in such a way that the orbit indexed by $\pi\in\NC$ is $W/W_\pi$ as a $W$-set.

\begin{lemm}  \label{lemm:pftonc}
    The map $\PF \to \NC$ defined by $w W_{\pi} \mapsto \pi$ is order-preserving and rank-preserving.  Moreover, its restriction to any interval $[W,\{w\}]$ is a poset isomorphism. 
\end{lemm}

\begin{proof}
    Note that the second part of the statement implies the first one.  So let $w\in W$, we show that the map $[W,\{w\}] \to \NC$, $wW_{\pi} \mapsto \pi$ is an isomorphism.

    First note that
    \[
        [W,\{w\}]
            =
        \big\{ wW_{\pi} \;:\; \pi\in\NC \big\}.
    \]
    Indeed, if $w'W_{\pi} \in [W,\{w\}]$, we have $w'W_{\pi} \leq \{w\}$, {\it i.e.}, $w \in w'W_{\pi}$, and it follows $w'W_{\pi} = wW_{\pi}$.  The previous lemma implies that the map $[W,\{w\}] \to \NC$, $wW_{\pi} \mapsto \pi$ is bijective.  Moreover, it preserves the order because the map $\pi \mapsto W_\pi$ does.
\end{proof}

Note that the previous lemma implies that the principal order ideal $\big\{ w' W_{\pi'} \in \PF \;:\; w' W_{\pi'} \leq wW_\pi \big\}$ is isomorphic to the interval $[c,\pi]$ in $\NC$.  The next lemma gives the structure of a principal order filter.

In the next statement, we use the notation $\PF(W_\pi)$ to indicate that we consider the poset $\PF$ defined with respect to the reflection group $W_\pi$.  (Similar notation will be used for $\NC$, $\CPF$, etc.)

\begin{lemm}  \label{lemm:pffilter}
    For each $wW_\pi \in \PF$, there is a natural poset isomorphism 
    \begin{align} \label{eq:isoordefilter}
        \big\{ w' W_{\pi'} \in \PF \;:\; w' W_{\pi'} \geq wW_\pi \big\}
        \to
        \PF(W_\pi).
    \end{align}
\end{lemm}

\begin{proof}   
    Via the action of $W$, we can assume $w=e$ in the left-hand side of~\eqref{eq:isoordefilter}.  The relation $w' W_{\pi'} \geq wW_\pi = W_\pi$ means $w' W_{\pi'} \subset W_\pi$.  In particular we get $w' \in W_{\pi}$ and  $\pi'\in W_\pi$ ({\it i.e.}, $ \pi' \geq \pi$), so $w' W_{\pi'} \in \PF(W_\pi)$.  It is straightforward to complete the proof.
\end{proof}

\begin{rema}
    Note that the group $w W_\pi w^{-1}$ naturally acts on the left-hand side of~\eqref{eq:isoordefilter}, and $W_\pi$ acts on the right-hand side.  It is easily seen that the characters of these actions are the same (if we identify the characters of $W_\pi$ and $w W_\pi w^{-1}$ in the natural way).  In the case $w=e$, the poset isomorphism in the lemma commutes with the action of $W_\pi$.
\end{rema}

Although we will not use the following property of being a lattice, it is worth writing it.

\begin{lemm}
    Let $wW_\pi,w'W_{\pi'}\in \PF$ such that $wW_\pi \cap w'W_{\pi'} \neq \varnothing$.  Then $wW_\pi \cap w'W_{\pi'} \in \PF$ and it is the least upper bound of $wW_\pi$ and $w'W_{\pi'}$ in $\PF$.
\end{lemm}


\begin{proof}
    Let $w'' \in wW_\pi \cap w'W_{\pi'}$, so that $wW_\pi = w''W_\pi$ and $w'W_{\pi'} = w'' W_{\pi'}$.  We have:
    \[
        wW_\pi \cap w'W_{\pi'}
            =
        w''W_\pi \cap w''W_{\pi'}
            =
        w''(W_\pi \cap W_{\pi'})
            =
        w'' W_{\pi\vee\pi'} \in \PF,
    \]
    by~\eqref{eq:joinnc}.  As the order is reverse inclusion, $wW_\pi \cap w'W_{\pi'}$ is clearly the least upper bound of $wW_\pi$ and $w'W_{\pi'}$. 
\end{proof}

\begin{prop}
    The poset $\hat{\PF}$ is a lattice.
\end{prop}

\begin{proof}
    It is clear from the previous lemma that joins in $\hat{\PF}$ always exist.  If we identify the extra top element $\hat 1 \in \hat{\PF}$ with $\varnothing$, a join of two elements is given by taking their intersection.  As $\hat{\PF}$ is bounded, it follows that meets also exist. 
\end{proof}

\begin{defi}
    For each integer $m\geq 1$, let $\park_m$ denote the character of $W$ acting on $m$-element multichains $w_1W_{\pi_1} \leq \dots \leq w_mW_{\pi_m}$ in $\PF$, {\it i.e.}:
    \[
        \park_m
            :=
        \bzeta(\PF,m).
    \]
    And for each integer $m\geq 1$, let $\park'_m$ denote the character of $W$ acting on $m$-element multichains $w_1W_{\pi_1} \leq \dots \leq w_mW_{\pi_m}$ in $\PF$ such that $w_1W_{\pi_1} \in \PF'$.
\end{defi}

We refer to~\cite{douvropoulosjosuatverges} for more context and results about these two characters, in particular for the analog of Equations~\eqref{eq:chpf} and~\eqref{eq:chpf'} in the context of finite real reflection groups.  As we focus here on the topology of cluster parking functions, these characters will only be used in the final step where we consider the characters of $W$ acting on the homology of cluster parking functions.



\begin{prop}  \label{prop:zetaPF}
    If $W$ is irreducible and its Coxeter number is $h$, we have for each $w\in W$:
    \begin{align*}
        \park_m(w) 
            &=
        (mh+1)^{n-\ell(w)},\\
        \park'_m(w) 
            &=
        (mh-1)^{n-\ell(w)}.
    \end{align*}
\end{prop}

\begin{proof}
    The first equality is a conjecture by Rhoades~\cite{rhoades}, who proved it for infinite families in the finite type classification (more precisely, his conjecture is more general because it involves an extra cyclic group).  The variant about prime parking functions has been introduced in~\cite{douvropoulosjosuatverges}, where it is proved that the two statements are equivalent.  Eventually, a full (case-free) proof has been announced in~\cite{douvropoulosjosuatverges2}.
\end{proof}

We finish this section with a conjecture.  See Proposition~\ref{prop:flag} for its relevance concerning cluster parking functions.

\begin{conj} \label{conj:helly}
    $\PF$ has the Helly property, {i.e.}: if $k\geq 3$ and $w_1 W_{\pi_1},\dots, w_k W_{\pi_k} \in \PF$ are such that $w_i W_{\pi_i} \cap w_j W_{\pi_j} \neq \varnothing$ for $1\leq i <j \leq n$, then we have $\cap_{i=1}^k w_i W_{\pi_i} \neq \varnothing$.
\end{conj}

\begin{rema}
    It is easily seen that $\hat{\PF}$ is an atomic lattice (every element is a join of atoms).  This follows from $\NC$ being atomic.  Using this, one easily shows that $\PF$ has the Helly property if and only if its subset of rank 1 elements has the Helly property. 
\end{rema}

The Helly numbers of cosets in finite groups have been investigated in \cite{domokos}.

\subsection{The cluster complex}

\label{sec:clus}

Let $\Phi = \Phi_+ \cup \Phi_-$ be a root system for $W$ (in the sense of Coxeter groups), let $\Pi\subset \Phi_+$ denote the set of simple roots, and for $\alpha\in \Phi$ we denote by $t_{\alpha}\in T$ the corresponding reflection.  Consider a set partition $\Pi = \Pi_\bullet \uplus \Pi_\circ$ where each term $\Pi_\bullet$ or $\Pi_\circ$ contains pairwise orthogonal elements, let $c_\bullet = \prod_{\alpha\in \Pi_\bullet} t_{\alpha}$ and $c_\circ = \prod_{\alpha\in \Pi_\circ} t_{\alpha}$.  The (standard) Coxeter element $c = c_\bullet c_\circ$ is called the {\it bipartite Coxeter element}.  

There are several equivalent definitions for the cluster complex, and a review sufficient for our purpose has been given in~\cite[Section~6]{douvropoulosjosuatverges}.  We give here the original definition of Fomin and Reading~\cite{fominreading}.

\begin{defi}[\cite{fominreading}]
    A {\it colored root} is a pair $(\alpha,i)\in \Phi \times \mathbb{N}$, denoted by $\alpha^i$.  Let $m\geq 1$ be an integer.  The set of $m$-{\it colored positive roots} is $\Phi^{(m)}_+ := \Phi_+ \times \llbracket 1,m\rrbracket$, and the set of $m$-{\it colored almost positive roots} is:
    \[
        \Phi^{(m)}_{\geq -1}
            :=
        \Phi^{(m)}_+
        \uplus
        \big\{ (-\alpha)^1 \;:\; \alpha\in\Pi \big\}.
    \]
The {\it rotation} $\mathcal{R}$ is a bijection of $\Phi^{(m)}_{\geq -1}$ onto itself, defined by:
\[
    \mathcal{R}(\alpha^i) = 
    \begin{cases}
        \alpha^{i+1}   &\text{ if } \alpha \in \Phi_+ \text{ and } i<m,\\
        (-\alpha)^1    &\text{ if } \alpha \in -\Pi_\bullet \text{ and } i=1, \text{ or } \alpha\in\Pi_\circ \text{ and } i=m, \\
        c(\alpha)^1    &\text{ if }  \alpha\in -\Pi_\circ \text{ and } i=1, \text{ or } \alpha \in \Phi_+ \backslash \Pi_\circ \text{ and } i=m.
    \end{cases}
\]
\end{defi}

\begin{defi}[\cite{fominreading}]
    The binary relation $\mathrel{\|}$ on $\Phi^{(m)}_{\geq -1}$ is defined by the following properties.
    \begin{itemize}
        \item $(-\alpha)^1 \mathrel{\|} \beta^i$ with $\alpha\in\Pi$ holds iff $\alpha$ has coefficient $0$ in the expansion of $\beta$ as a linear combination of simple roots,
        \item $\alpha^i \mathrel{\|} \beta^j \Longleftrightarrow \mathcal{R}(\alpha^i) \mathrel{\|} \mathcal{R}(\beta^j)$, for all $\alpha^i, \beta^j \in \Phi^{(m)}_{\geq -1}$.
    \end{itemize}
    The simplicial complex $\Delta^{(m)}$ is defined as the flag complex with vertex set $\Phi^{(m)}_{\geq -1}$ associated to the binary relation $\mathrel{\|}$, and $\Delta^{(m),+}$ is the full subcomplex with vertex set $\Phi^{(m)}_+$.
\end{defi}

The definition of the binary relation immediately gives some useful properties:
\begin{itemize}
    \item For $\alpha\in\Pi$, the link of $(-\alpha)^1$ in $\Delta^{(m)}$ is $\Delta^{(m)}(W_{S\backslash\{t_\alpha\}})$.
    \item The self-map on $\Delta^{(m)}$ induced by $\mathcal{R}$ is an automorphism.
\end{itemize}

\begin{rema}
It is natural to include the case $m=0$ by declaring that $\Delta^{(0)}$ is an $(n-1)$-simplex with vertex set $-\Pi$.  The properties of $\Delta^{(m)}$ generally hold for $m=0$ as well.  It might seem natural to declare that $\Delta^{(0),+}$ is the empty complex, however this will not reflect the properties of $\Delta^{(m),+}$ in general.  In what follows, we assume $m\geq 1$ and discuss explicitly the case $m=0$.
\end{rema}

An alternative definition of the complex $\Delta^{(m)}$ was given by Athanasiadis and Tzanaki~\cite{athanasiadistzanaki}, and a useful consequence is the following: 

\begin{prop}[\cite{athanasiadistzanaki}]
    For each $f\in \Delta^{(m)}$, there is an indexing $f = \{\alpha_1^{i_1},\dots,\alpha_k^{i_k}\}$ such that $\prod f := \prod_{i=1}^k t_{\alpha_i}$ is in $\NC$.  Moreover $\prod_{i=1}^k t_{\alpha_i}$ is a reduced factorization.  
\end{prop}

When $\NC$ is defined with respect to the bipartite Coxeter element, define for each $\pi\in\NC$:
\[
    \pi^\kappa := c_\bullet \pi c_\circ.
\]
The map $\pi \mapsto \pi^\kappa$ is an involutive anti-automorphism of $\NC$, called the {\it bipartite Kreweras complement} (see~\cite{bianejosuatverges}).

\begin{defi}[\cite{douvropoulosjosuatverges}]
    For $f\in \Delta^{(m)}$, we define
    \[
        \underline{f} := \big( \prod f \big)^\kappa.
    \]
    As the Kreweras complement of $\prod f \in \NC$, we have $ \underline{f} \in \NC$.
\end{defi}

As explained in~\cite{douvropoulosjosuatverges}, it is the conjugacy class of $\underline{f}$ that is of primary interest.  In type $A_{n-1}$, the conjugacy class of $\underline{f}$ is an integer partition of $n$, defined as the cyclic type of $\underline{f}$ as an element of the symmetric group. 
 We will see that it corresponds to the integer partition $\lambda$ defined in the introduction (see Lemma~\ref{lemm:bijtypeA}).

\begin{lemm}
    The map $f \mapsto \underline{f}$ is order-preserving.
\end{lemm}

\begin{proof}
    If $f'\subset f$, by definition $\prod f'$ is a subword of $\prod f$.  By definition of the absolute order, this shows that the map $f\mapsto \prod f$ is order-reversing.  Since the Kreweras complement is also order-reversing, we get the result. 
\end{proof}

It was proved in~\cite{fominreading} that the link of a face $f\in \Delta^{(m)}$ is isomorphic to $\Delta^{(m)}(W_I)$ for some $I\subset S$.  The following is a more precise statement:

\begin{prop}[{\cite[Remark~6.15]{douvropoulosjosuatverges}}] \label{prop:linkdelta}
    The link of $f$ in $\Delta^{(m)}$ is isomorphic to $\Delta^{(m)}(W_{\underline{f}})$.
\end{prop}

\begin{proof}
    The proof was just sketched in~\cite{douvropoulosjosuatverges}, so let us write some details.  First assume that $f$ contains $(-\alpha)^1$ where $\alpha\in \Pi$.  Since the link of $\{(-\alpha)^1\}$ in $\Delta^{(m)}$ identifies with $\Delta^{(m)}(W_{S\backslash\{t_\alpha\}})$, the link of $f$ in $\Delta^{(m)}$ identifies with the link of $f\backslash\{-\alpha\}$ in $\Delta^{(m)}(W_{S\backslash\{t_\alpha\}})$.  Moreover, we have
    \[
        \prod(f\backslash\{-\alpha\}) 
        =
        \begin{cases}
            t_\alpha (\prod f) & \text{ if } \alpha\in \Pi_\bullet, \\
            (\prod f) t_\alpha & \text{ if } \alpha\in \Pi_\circ.
        \end{cases}
    \]
    It follows that $\underline{f\backslash\{-\alpha\}} = \underline{f}$ (where in the left-hand side, the bipartite Kreweras complement is computed inside $W_{S\backslash\{t_\alpha\}}$).  The conclusion is that when $f$ contains an element of $-\Pi$, the result holds via an induction on the rank.  

    Now assume $f\cap (-\Pi) = \varnothing$.  There is an integer $i$ such that $\mathcal{R}^i(f) \cap  (-\Pi) \neq \varnothing$ (see~\cite{fominreading}).  Moreover, it has been shown in~\cite[Proposition~11.1]{douvropoulosjosuatverges} that $\underline{f}$ and $\underline{\mathcal{R}(f)}$ are conjugate in $W$.  These two properties show that the general case where $f\cap (-\Pi) = \varnothing$ follows from the previous case.
\end{proof}




\begin{prop} \label{prop:subcompl}
For each $\pi\in \NC$, the subcomplexes
\begin{align}
    \Delta^{(m),+}(\pi) 
        :=
    \big\{ 
            f\in \Delta^{(m),+} \;:\; \underline{f} \leq \pi 
    \big\},
            \qquad 
    \Delta^{(m)}(\pi) 
        :=
    \big\{ 
            f\in \Delta^{(m)} \;:\; \underline{f} \leq \pi 
    \big\}
\end{align}
of $\Delta^{(m)}$ are pure $(\rk(\pi)-1)$-dimensional and shellable. 
\end{prop}

\begin{proof}
    These results are essentially in Athanasiadis and Tzanaki~\cite{athanasiadistzanaki} (the first one is explicitly stated, and the second one is a straightforward consequence of their methods and results).  For the sake of thoroughness, let us give some details.  
    
    First note that since $f\mapsto \underline{f}$ is order-preserving, both $\Delta^{(m),+}(\pi)$ and $\Delta^{(m)}(\pi)$ are order ideals in the face poset of $\Delta^{(m)}$ ({\it i.e.}, subcomplexes of $\Delta^{(m)}$).  For $I\subset -\Pi$, define:
    \[
        \Gamma(\pi,I)
            :=
        \big\{ 
            f\in \Delta^{(m)} \;:\; \underline{f} \leq \pi \text{ and } f\cap (-\Pi) \subset I
        \big\}.
    \]
    It is again immediate to check that this is a subcomplex of $\Delta^{(m)}$.  By induction on $\#I$ and the rank of $W$, we show that these subcomplexes are all purely $(n-\ell(\pi)-1)$-dimensional and shellable.  The proposition will follow from the cases $I=\varnothing$ and $I=-\Pi$, respectively.  In the case where $I = \varnothing$, the result is precisely~\cite[Proposition~3.1,~(i)]{athanasiadistzanaki}.  So, now we assume $I\neq\varnothing$ and write $I = I' \uplus\{ \alpha\}$ for some $\alpha\in -\Pi$.  
    
    First note that if $\alpha$ is a vertex of $f \in \Gamma(\pi,I)$, we get $t_{\alpha} \geq \prod f \geq \pi^\kappa$.  So, in the case where $t_{\alpha} \not\geq \pi^\kappa$, $\alpha$ cannot be a vertex of $f\in \Gamma(\pi,I)$. So we have $\Gamma(\pi,I)= \Gamma(\pi,I')$, and the result holds by the induction hypothesis.
    
    Now consider the case $t_{\alpha} \geq \pi^\kappa$.  We examine two subcomplexes of $\Gamma(\pi,I)$.
    \begin{itemize}
        \item The {\it deletion} of the vertex $\alpha$ in $\Gamma(\pi,I)$: this is  $\Gamma(\pi,I')$, by definition.  By the induction hypothesis, it is purely $(n-\ell(\pi)-1)$-dimensional and shellable.
        \item The {\it link} of the vertex $\alpha$ in $\Gamma(\pi,I)$: it can be described as follows.  Let $W'$ be the maximal standard parabolic subgroup of $W$ generated by $S-\{t_{\alpha} \}$.  Its rank is $n'=n-1$.  We have either $\pi = t_{\alpha} \pi' $ with $\pi' \in \NC(W')$ (if $\alpha \in -\Pi_\bullet$), or $\pi = \pi' t_{\alpha}$ with $\pi' \in \NC(W')$ (if $\alpha \in -\Pi_\circ$).  In each case, we can directly identify $\Gamma(\pi,I)$ with $\Gamma(\pi',I')$, where the latter is defined in the cluster complex attached to $W'$.  By the induction hypothesis, it is purely $(n'-\ell(\pi)-1)$-dimensional and shellable.
    \end{itemize} 
    From the properties of these two subcomplexes, we deduce that $\Gamma(\pi,I)$ is also purely $(n-\ell(\pi)-1)$-dimensional and shellable (see~\cite[Lemma~2.1,~(i)]{athanasiadistzanaki}, for example).
\end{proof}

We end this section by stating a result from our previous work~\cite{douvropoulosjosuatverges}.  Define the sign character $\sign$ of $W$ by
\[
    \sign(w) := (-1)^{\ell(w)}.
\]
For each parabolic subgroup $P$, there is a left action of $W$ on the quotient $W/P$.  The character of this action is denoted by $\ch(W/P)$.  In the following result, we consider an alternating sum in the character ring of $W$.  Essentially, the identities below come from a refined enumeration of faces of $\Delta^{(m)}$ and $\Delta^{(m),+}$ according to the conjugacy class of $\underline{f}$ (see~\cite{douvropoulosjosuatverges} for details).

\begin{prop}[{\cite[Section~12, Equations~(72,73)]{douvropoulosjosuatverges}}] \label{prop:altersumcharac}
    We have:
    \begin{align}
        \sum_{f\in \Delta^{(m)}} (-1)^{\dim(f)} \ch(W/W_{\underline{f}})
            =
        (-1)^{n-1} \sign \otimes \park_m,\\
        \sum_{f\in \Delta^{(m),+}} (-1)^{\dim(f)} \ch(W/W_{\underline{f}})
            =
        (-1)^{n-1} \sign \otimes \park'_m.        
    \end{align}
\end{prop}

This is the analog of Equations~\eqref{altern_sum} and~\eqref{altern_sum2} in the introduction.

\subsection{Combinatorial models in terms of dissections of polygons}

Let us describe explicitly the combinatorics of type $A_{n-1}$ cluster complex, following Fomin and Zelevinsky~\cite{fominzelevinsky}.  As mentioned in the introduction, vertices are diagonals of an $(n+2)$-gon.  Let $(e_i)_{1\leq i \leq n}$ be the canonical basis of $\mathbb{R}^n$.  The indexing is characterized (up to an order 2 symmetry of the polygon) by the following properties.
\begin{itemize}
    \item Negative simple roots $(-e_i+e_{i+1})_{1\leq i \leq n-1}$ form the ``snake'' triangulation
    \[
        \begin{tikzpicture}
            \draw[very thick] (0.38,-0.92) -- (-0.38,-0.92) -- (-0.92,-0.38) -- (-0.92,0.38) -- (-0.38,0.92) -- (0.38,0.92) -- (0.92,0.38) -- (0.92,-0.38) -- (0.38,-0.92);
            \draw[very thick,rounded corners=1] (-0.38,0.92) -- (-0.92,-0.38) -- (0.38,0.92) -- (-0.38,-0.92) -- (0.92,0.38) -- (0.38,-0.92);
        \end{tikzpicture},
    \]
    in such a way that two diagonals indexed by $-e_i+e_{i+1}$ and $-e_j+e_{j+1}$ share an endpoint if and only if $j-i = \pm 1$.
    \item The diagonal indexed by a positive root $e_i-e_j$ (with $i<j$) crosses the diagonal indexed $-e_k+e_{k+1}$ if and only if $i\leq k < j$.
\end{itemize}
For example, 
\begin{align} \label{eq:exf1f2}
    \begin{tikzpicture}
            \draw[very thick] (0.38,-0.92) -- (-0.38,-0.92) -- (-0.92,-0.38) -- (-0.92,0.38) -- (-0.38,0.92) -- (0.38,0.92) -- (0.92,0.38) -- (0.92,-0.38) -- (0.38,-0.92);
            \draw[very thick,rounded corners=1] (-0.38,0.92) -- (-0.92,-0.38) -- (0.38,0.92);
            \draw[very thick,rounded corners=1] 
            (-0.92,-0.38) -- (0.92,-0.38) -- (-0.38,-0.92);
    \end{tikzpicture}
    \quad 
    \text{ and }
    \quad
        \begin{tikzpicture}
            \draw[very thick] (0.38,-0.92) -- (-0.38,-0.92) -- (-0.92,-0.38) -- (-0.92,0.38) -- (-0.38,0.92) -- (0.38,0.92) -- (0.92,0.38) -- (0.92,-0.38) -- (0.38,-0.92);
            \draw[very thick,rounded corners=1] (-0.38,-0.92) -- (-0.38,0.92) -- (0.38,-0.92);
    \end{tikzpicture}
\end{align}
respectively correspond to the faces
\[
    f_1 = \{
        -e_1+e_2,-e_2+e_3, e_3-e_6, e_5-e_6
    \} \in \Delta_6
    \text{ and }
    f_2 = \{
        e_2-e_3, e_2-e_5
    \} \in \Delta_6.
\]
See~\cite{fominreading,fominzelevinsky} for detailed combinatorial descriptions of $\Delta^{(m)}$ and $\Delta^{(m),+}$ in types $A$, $B$, {\it etc}.

Let us go on with describing the noncrossing partition $\prod f \in \NC$ in type $A$.  To this end, we label $n$ vertices on a circle with $1,\dots,n$ in such a way that going counterclockwise we read the labels $1,3,5,\dots,n-1,n,n-2,\dots,4,2$ (if $n$ is even) or $1,3,5,\dots,n-2,n,n-1,\dots,4,2$ (if $n$ is odd).  These labels come from reading the bipartite Coxeter element as a cycle.  A {\it noncrossing forest} is obtained from $f$ by drawing an edge between $j$ and $k$ if $f$ contains an element $\alpha^i$ such that $t_\alpha$ is the transposition $(j,k)\in\mathfrak{S}_n$.  For example, the faces $f_1$ and $f_2$ as above respectively give
\begin{align}  \label{eq:exf1f2nc}
    \begin{tikzpicture}
        \draw[very thick] (0,0) circle (1);
        \node at (-0.54,1.1) {$2$};
        \node at (0.54,1.1) {$4$};
        \node at (1.2,0) {$6$};
        \node at (0.54,-1.1) {$5$};
        \node at (-0.54,-1.1) {$3$};
        \node at (-1.2,0) {$1$};
        \draw[very thick] (-1,0) -- (-0.5,0.866);
        \draw[very thick] (-0.5,0.866) -- (-0.5,-0.866);
        \draw[very thick] (-0.5,-0.866) -- (1,0);
        \draw[very thick] (1,0) -- (0.5,-0.866);        
    \end{tikzpicture}
    \text{ and }
        \begin{tikzpicture}
        \draw[very thick] (0,0) circle (1);
        \node at (-0.54,1.1) {$2$};
        \node at (0.54,1.1) {$4$};
        \node at (1.2,0) {$6$};
        \node at (0.54,-1.1) {$5$};
        \node at (-0.54,-1.1) {$3$};
        \node at (-1.2,0) {$1$};
        \draw[very thick] (-0.5,-0.866) -- (-0.5,0.866) -- (0.5,-0.866);
    \end{tikzpicture}.
\end{align}
The noncrossing partition $\prod f$ is obtained by taking the connected components of this noncrossing forest, so 
\[
    \prod f_1 = \{ \{1,2,3,5,6\},\{4\}\},
    \quad
    \prod f_2 = \{ \{1\}, \{2,3,5\},\{4\},\{6\}\}.
\]
Now, write $n$ additional labels $1',\dots,n'$ on the circle so that reading it clockwise gives $1,2',3,4',\dots 3',2,1'$.  The bipartite Kreweras complement of the noncrossing partition is obtained by connecting $i'$ with $j'$ by a path whenever it is possible to do so without creating a crossing, and looking at the connected components created by these paths.  In the example above, the pictures become 
\begin{align} \label{eq:exf1f2krew}
    \begin{tikzpicture}
        \draw[very thick,color=red] (0,1) -- (0.866,0.5);
        \draw[very thick] (0,0) circle (1);
        \node at (-0.54,1.1) {$2$};
        \node at (0.54,1.1) {$4$};
        \node at (1.2,0) {$6$};
        \node at (0.54,-1.1) {$5$};
        \node at (-0.54,-1.1) {$3$};
        \node at (-1.2,0) {$1$};
        \node at (1.1,-0.54) {\textcolor{red}{$6'$}};
        \node at (1,0.7) {\textcolor{red}{$5'$}};
        \node at (0,1.2) {\textcolor{red}{$3'$}};
        \node at (-1,0.7) {\textcolor{red}{$1'$}};
        \node at (-1.1,-0.54) {\textcolor{red}{$2'$}};
        \node at (0,-1.2) {\textcolor{red}{$4'$}};
        \draw[very thick] (-1,0) -- (-0.5,0.866);
        \draw[very thick] (-0.5,0.866) -- (-0.5,-0.866);
        \draw[very thick] (-0.5,-0.866) -- (1,0);
        \draw[very thick] (1,0) -- (0.5,-0.866);        
    \end{tikzpicture}
    \text{ and }
        \begin{tikzpicture}
        \draw[very thick,color=red] (-0.866,0.5) -- (-0.866,-0.5);
        \draw[very thick,color=red] (0,1) -- (0.866,0.5) -- (0.866,-0.5) -- (0,1);
        \draw[very thick] (0,0) circle (1);
        \node at (-0.54,1.1) {$2$};
        \node at (0.54,1.1) {$4$};
        \node at (1.2,0) {$6$};
        \node at (0.54,-1.1) {$5$};
        \node at (-0.54,-1.1) {$3$};
        \node at (-1.2,0) {$1$};
        \node at (1.1,-0.54) {\textcolor{red}{$6'$}};
        \node at (1,0.7) {\textcolor{red}{$5'$}};
        \node at (0,1.2) {\textcolor{red}{$3'$}};
        \node at (-1,0.7) {\textcolor{red}{$1'$}};
        \node at (-1.1,-0.54) {\textcolor{red}{$2'$}};
        \node at (0,-1.2) {\textcolor{red}{$4'$}};
        \draw[very thick,rounded corners=1] (-0.5,-0.866) -- (-0.5,0.866) -- (0.5,-0.866);
    \end{tikzpicture}.
\end{align}
This illustrates that
\[
    \underline{f_1} = \{\{1\},\{2\},\{3,5\},\{4\},\{6\}\}
    \text{ and }
    \underline{f_2} = \{\{1,2\},\{3,5,6\},\{4\}\},
\]
which can also be checked in $\mathfrak{S}_n$ by unfolding the definitions of the general case. 

To complete this combinatorial description of type $A$, we have the following:  

\begin{lemm} \label{lemm:bijtypeA}
    Let $f \in \Delta^{(m)}_n$, viewed as a dissection of an $(mn+2)$-gon.  The integer partition $\lambda \vdash n$ associated to $f$ (as defined in the introduction) is the decreasing sort of block sizes of $\underline{f}$.  
    
\end{lemm}

\begin{proof}
    Consider the link of $f \in \Delta^{(m)}_n$ (viewed as a dissection of an $(mn+2)$-gon).  Clearly, faces $f' \in \Delta^{(m)}_n $ containing $f$ are obtained by choosing a dissection of each inner polygon defined by $f$.  This shows that the link of $f$ is isomorphic to a join
    \begin{equation} \label{eq:joinA}
        \Delta^{(m)}_{\lambda_1} \ast
        \Delta^{(m)}_{\lambda_2} \ast \cdots.
    \end{equation}
    
    On the other hand, recall from Proposition~\ref{prop:linkdelta} that the link of $f$ is $\Delta^{(m)}(W_{\underline{f}})$ which is the join of the complexes $\Delta^{(m)}(W_i)$ where $W_i$ are the irreducible factors of $W_{\underline{f}}$.  As we focus on type $A$, it means that this link is the join 
        \begin{equation} \label{eq:joinA2}
        \Delta^{(m)}_{\mu_1} \ast
        \Delta^{(m)}_{\mu_2} \ast \cdots
    \end{equation}
    where the partition $\mu \vdash n$ is given by block sizes of $\underline{f}$.

    The result $\lambda = \mu$ follows from the absence of exceptional isomorphisms: $\Delta^{(m)}(W) \simeq \Delta^{(m)}(W')$ implies that $W$ and $W'$ are Coxeter-isomorphic (see~\cite[Theorem~2.12]{josuatverges}).
\end{proof}

Another way to prove the previous lemma is to give an explicit bijection between inner $(mk+2)$-gons of the dissections and blocks of size $k$ in $\underline{f}$, for each $k\geq 1$.  We omit details and only give an example (in the case $m=1$) to illustrate a general construction.  Consider the inner pentagon in the second dissection in~\eqref{eq:exf1f2}. Its diagonals are indexed by the roots $e_2-e_4$, $e_2-e_6$, $e_4-e_5$, $e_4-e_6$, $-e_5+e_6$.  The associated reflections are $(24)$, $(26)$, $(45)$, $(46)$, $(56)$.  By dividing the disk into components separated by edges of the noncrossing tree in~\eqref{eq:exf1f2nc}, we see that all these reflections lie in the same component.  This component corresponds to the block $\{3,5,6\}$ (see~\eqref{eq:exf1f2krew}).  Thus, the inner pentagon we started from is associated to the block $\{3,5,6\}$.  

The construction of the previous paragraph doesn't work when $k=1$, because we get an empty set of reflections.  To overcome this, consider the triangle as an intersection of inner polygons from smaller faces, for example:
\begin{align*}
    \begin{tikzpicture}
            \draw[very thick] (0.38,-0.92) -- (-0.38,-0.92) -- (-0.92,-0.38) -- (-0.92,0.38) -- (-0.38,0.92) -- (0.38,0.92) -- (0.92,0.38) -- (0.92,-0.38) -- (0.38,-0.92);
            \draw[very thick,rounded corners=1,fill=lightgray] (-0.38,-0.92) -- (-0.38,0.92) -- (0.38,-0.92) -- (-0.38,-0.92);
    \end{tikzpicture}
    \quad = \quad
    \begin{tikzpicture}
            \draw[very thick] (0.38,-0.92) -- (-0.38,-0.92) -- (-0.92,-0.38) -- (-0.92,0.38) -- (-0.38,0.92) -- (0.38,0.92) -- (0.92,0.38) -- (0.92,-0.38) -- (0.38,-0.92);
            \draw[very thick,rounded corners=1,fill=lightgray] (-0.38,0.92) -- (0.38,-0.92) -- (-0.38,-0.92) -- (-0.92,-0.38) -- (-0.92,0.38) -- (-0.38,0.92);
    \end{tikzpicture}
    \quad \cap \quad 
    \begin{tikzpicture}
        \draw[very thick] (0.38,-0.92) -- (-0.38,-0.92) -- (-0.92,-0.38) -- (-0.92,0.38) -- (-0.38,0.92) -- (0.38,0.92) -- (0.92,0.38) -- (0.92,-0.38) -- (0.38,-0.92);
        \draw[very thick,rounded corners=1,fill=lightgray](-0.38,-0.92) -- (-0.38,0.92) -- (0.38,0.92) -- (0.92,0.38) -- (0.92,-0.38) -- (0.38,-0.92) -- (-0.38,-0.92);
    \end{tikzpicture}.
\end{align*}
We expect that the same relation holds for the associated blocks of $\underline{f}$.  By computing those in the right-hand side, we respectively get $\{1,2,4\}$ and $\{3,4,5,6\}$.  So the block associated to the triangle is:
\[
    \{4\}
    =
    \{1,2,4\} \cap \{3,4,5,6\}.
\]

Let us now describe the combinatorics in the case where $n=2$, {\it i.e.}, $W$ is a dihedral group $I_2(k)$ for some $k\geq 2$.  There is a natural embedding $\Delta^{(m)}(I_2(k)) \subset \Delta^{(m)}(\mathfrak{S}_{k})$ (though only as a subposet rather than as a subcomplex, as dimension is not preserved).  All this relies on the embedding $I_2(k) \subset \mathfrak{S}_k$.  A facet $f$ of $\Delta^{(m)}(I_2(k))$ is an $(m+2)$-angulation of the regular $(mk+2)$-gon such that:
\begin{itemize}
    \item the adjacency graph of its inner polygons is the path graph on $k$ vertices,
    \item by taking every other diagonal (every other edge in the adjacency graph), we get a set of pairwise parallel diagonals.
\end{itemize}
A vertex of this facet $f$ is a dissection obtained by keeping every other diagonal.  For example, a facet of $\Delta^{(3)}(I_2(5))$ is the quintangulation of the heptadecagon 
\[
    \begin{tikzpicture}[scale=1.4]
        \draw[very thick] (1.00, 0.000) -- (0.932, 0.361) -- (0.739, 0.674) -- (0.446, 0.895) -- (0.0923, 0.996) -- (-0.274, 0.962) -- (-0.603, 0.798) -- (-0.850, 0.526) -- (-0.983, 0.184) -- (-0.983, -0.184) -- (-0.850, -0.526) -- (-0.603, -0.798) -- (-0.274, -0.962) -- (0.0923, -0.996) -- (0.446, -0.895) -- (0.739, -0.674) -- (0.932, -0.361) -- (1.00,0.00);
        \draw[very thick,rounded corners=1] (1.00, 0.000) -- (0.0923, 0.996);
        \draw[very thick,rounded corners=1] (0.932, -0.361) -- (-0.603, 0.798);
        \draw[very thick,rounded corners=1] (-0.850, 0.526) -- (0.446, -0.895);
        \draw[very thick,rounded corners=1] 
        (0.0923, -0.996) -- (-0.983, -0.184);
    \end{tikzpicture},
\]
and its two vertices are the dissections:
\[
     \begin{tikzpicture}[scale=1.4]
        \draw[very thick] (1.00, 0.000) -- (0.932, 0.361) -- (0.739, 0.674) -- (0.446, 0.895) -- (0.0923, 0.996) -- (-0.274, 0.962) -- (-0.603, 0.798) -- (-0.850, 0.526) -- (-0.983, 0.184) -- (-0.983, -0.184) -- (-0.850, -0.526) -- (-0.603, -0.798) -- (-0.274, -0.962) -- (0.0923, -0.996) -- (0.446, -0.895) -- (0.739, -0.674) -- (0.932, -0.361) -- (1.00,0.00);
        \draw[very thick,rounded corners=1] (0.932, -0.361) -- (-0.603, 0.798);
        \draw[very thick,rounded corners=1] 
        (0.0923, -0.996) -- (-0.983, -0.184);
    \end{tikzpicture},  \qquad
     \begin{tikzpicture}[scale=1.4]
        \draw[very thick] (1.00, 0.000) -- (0.932, 0.361) -- (0.739, 0.674) -- (0.446, 0.895) -- (0.0923, 0.996) -- (-0.274, 0.962) -- (-0.603, 0.798) -- (-0.850, 0.526) -- (-0.983, 0.184) -- (-0.983, -0.184) -- (-0.850, -0.526) -- (-0.603, -0.798) -- (-0.274, -0.962) -- (0.0923, -0.996) -- (0.446, -0.895) -- (0.739, -0.674) -- (0.932, -0.361) -- (1.00,0.00);
        \draw[very thick,rounded corners=1] (1.00, 0.000) -- (0.0923, 0.996);
        \draw[very thick,rounded corners=1] (-0.850, 0.526) -- (0.446, -0.895);
    \end{tikzpicture}.
\]  
We omit the details, but it is straightforward to see that this model is equivalent to the one given in~\cite[Section~4]{fominreading}.

\section{Cluster parking functions}
\label{sec:cpf}

\subsection{Definition and properties}  We announced that cluster parking functions form a simplicial complex, however it is convenient to introduce them as a poset (recall that we identify a simplicial complex with its face poset). 

\begin{defi}
The $W$-poset of {\it $m$-cluster parking functions}, denoted by $\CPF^{(m)}$, is defined as follows:
\begin{itemize}
    \item its elements are the pairs $(wW_{\underline{f}} , f) \in \PF \times \Delta^{(m)}$,
    \item the order relation is inherited from the product order on $ \PF \times \Delta^{(m)}$,
    \item and the action of $W$ is inherited from the left action of $W$ on $\PF$.
\end{itemize}
Moreover, let $\CPF^{(m),+}$ denote the set of {\it positive cluster parking functions}, {\it i.e.}, those $(wW_{\underline{f}} , f)$ such that $f \in \Delta^{(m),+}$.
\end{defi}

Note that the poset $\CPF^{(m)}$ can be viewed as the fiber product of $\PF$ and $\Delta^{(m)}$ over $\NC$, along the order-preserving maps
\[
    wW_\pi \mapsto \pi,
    \quad \text{and}
    \quad
    f \mapsto \underline{f}.
\]
This point of view will be helpful in Section~\ref{sec:topcpf}, where we use a fiber poset theorem. 

\begin{rema} \label{rema:coxeter}
In the case $m=0$, $\Delta^ {(0)}$ contains a unique facet and its faces are in bijection with standard parabolic subgroups via $f\mapsto W_{\underline{f}}$.  So $\CPF^{(0)}$ (endowed with the action of $W$) naturally identifies to the {\it Coxeter complex} of $W$.
\end{rema}

\begin{lemm} \label{lemm:intervalbool}
    The order ideal generated by any element in $\CPF^{(m)}$ is a Boolean lattice. 
\end{lemm}

It follows that the poset $\CPF^{(m)}$ is ranked, and the rank of $(w W_{\underline{f}} , f )$ is $\# f = 1+\dim(f)$.

\begin{proof}
    Let $(wW_{\underline{f}},f) \in \CPF^{(m)}$.  A smaller element can be written $(w'W_{\underline{f'}},f')$ with $f'\subset f$ and $w'W_{\underline{f'}} \supset wW_{\underline{f}}$, since the order is inherited from $ \PF(W) \times \Delta^{(m)}$.  By Lemma~\ref{lemm:pftonc}, we can choose $w'=w$.  These smaller elements can thus be identified to the subsets of $f$ via the map $(w'W_{\underline{f'}},f') \mapsto f'$.
\end{proof}

\begin{prop}
    The poset $\CPF^{(m)}$ is the face poset of a simplicial complex. 
\end{prop}

\begin{proof}
    Let $X$ be the set of rank 1 elements in $\CPF^{(m)}$ (which should be thought of as the vertices of the underlying simplicial complex).  In view of Lemma~\ref{lemm:intervalbool}, it remains only to prove that the map 
    \begin{align}
        (wW_{\underline{f}},f) \mapsto \big\{ v\in X \;:\; v\leq (wW_{\underline{f}},f) \big\}
    \end{align}
    is injective.  Indeed, this map will then permit to identify each element of $\CPF^{(m)}$ with a subset of $X$, providing explicitly the structure of a simplicial complex with $X$ as its vertex set.

    So, consider $(wW_{\underline{f}},f) \in \CPF^{(m)}$, and write $f = \{ \alpha_1 , \dots , \alpha_k \}$.  The rank $1$ elements below $(wW_{\underline{f}}, f )$ are $(wW_{\underline{\{\alpha_i\}}} , \{\alpha_i\} )$ for $1\leq i \leq k$ (this follows from the bijection in the proof of the previous lemma).  To show injectivity, we need to recover $(wW_{\underline{f}}, f )$ from these rank 1 elements, which is done as follows:
    \begin{itemize}
        \item $f = \cup_{i=1}^k \{\alpha_i\}$,
        \item $wW_{\underline{f}} = \bigcap_{i=1}^k wW_{\underline{\{\alpha_i\}}}.$
    \end{itemize}
    Let us justify the latter point.  First note that $\prod f = t_{\alpha_1} \cdots t_{\alpha_k} $ is the meet in $\NC$ of the elements $t_{\alpha_i}$ ($1\leq i \leq k$).  Since the Kreweras complement is an anti-automorphism, $\underline{f}$ is the join of the elements $\underline{\{\alpha_i\}}$ ($1\leq i \leq k$).  By~\eqref{eq:joinnc}, we also have 
    \[
        W_{\underline{f}}
            =
        \bigcap_{i=1}^k W_{\underline{\{\alpha_i\}}}
    \]
    and the result follows.
\end{proof}

\begin{prop} \label{prop:flag}
    Assume that Conjecture~\ref{conj:helly} holds ({\it i.e.}, $\PF$ has the Helly property).  Then $\CPF^{(m)}$ is a flag simplicial complex.  (And consequently, $\CPF^{(m),+}$ is also flag.)
\end{prop}

\begin{proof}
    Let $( w_i W_{\underline{\{\alpha_i\}}} , \{\alpha_i\})_{1 \leq i \leq k}$ be a family of vertices in $\CPF^{(m)}$.  Assume that each pair of vertices form an edge in $\CPF^{(m)}$.  By the previous result and its proof, this means that for $1\leq i<j\leq n$, we have:
    \begin{itemize}
        \item $\{\alpha_i,\alpha_j\}$ is an edge in $\Delta^{(m)}$,
        \item $w_i W_{\underline{\{\alpha_i\}}} \cap w_j W_{\underline{\{\alpha_j\}}} = w_{i,j}  W_{\underline{\{\alpha_i, \alpha_j\}}} $ for some $w_{i,j} \in W$.
    \end{itemize}
    The first property above imply that $f = \{\alpha_1,\dots, \alpha_k\} \in \Delta^{(m)}_{k-1}$, since $\Delta^{(m)}$ is a flag simplicial complex.  Moreover, $\underline{f}$ is the join in $\NC$ of the elements $\underline{\{\alpha_i\}}$, $1\leq i \leq k$.
    
    The second property above implies $w_i W_{\underline{\{\alpha_i\}}} \cap w_j W_{\underline{\{\alpha_j\}}} \neq \varnothing$.  By the Helly property, $\cap_{i=1}^k w_i W_{\underline{\{\alpha_i\}}} \neq \varnothing$.  Let $w\in \cap_{i=1}^k w_i W_{\underline{\{\alpha_i\}}}$.  Using the isomorphism between the interval $[W,\{w\}]$ and $\NC$, we see that $\cap_{i=1}^k w_i W_{\underline{\{\alpha_i\}}} = w W_{\underline{f}}$. 

    By construction, the vertices of the simplex $(wW_{\underline{f}},f) \in \CPF^{(m)}_{k-1}$ are $( w_i W_{\underline{\{\alpha_i\}}} , \{\alpha_i\})_{1 \leq i \leq k}$.  This proves the flag property.
\end{proof}

\begin{prop} \label{prop:cpflink}
    The link of $(wW_{\underline{f}},f)$ in 
   $\CPF^{(m)}$ is naturally isomorphic to $\CPF^{(m)}(W_{\underline{f}})$.
\end{prop}

\begin{proof}
    By definition of the order relations, this link is 
    \[
        \big\{ (w'W_{\underline{f'}},f') 
            \;:\;
            w'W_{\underline{f'}} \subset wW_{\underline{f}} \text{ and } f \subset f'
        \big\}.
    \]
    By Lemma~\ref{lemm:pffilter} and Proposition~\ref{prop:linkdelta}, this naturally identifies to a subposet of $\PF(W_{\underline{f}}) \times \Delta^{(m)}(W_{\underline{f}})$.  It is straightforward to complete the proof.
\end{proof}

\begin{prop}
    If $W$ is reducible and $W\simeq W_1\times W_2$, we have $\CPF^{(m)}(W) \simeq \CPF^{(m)}(W_1) \ast \CPF^{(m)}(W_2)$ where $\ast$ is the join operation on simplicial complexes.
\end{prop}

The proof is straightforward and omitted.  This is analog to the corresponding statement on $\Delta^{(m)}$~\cite{fominreading}. 

\begin{prop} \label{prop:bchicpf}
    We have:
    \begin{align} \label{eq:chitildecpf}
        \tilde \bchi(\CPF^{(m)}) 
            = 
        (-1)^{n-1} \sign \otimes \park_m, \\
        \tilde \bchi(\CPF^{(m),+}) 
            = 
        (-1)^{n-1} \sign \otimes \park'_m.
    \end{align}
\end{prop}

\begin{proof}
    By definition of $\CPF^{(m)}$, its orbits are in bijection with $\Delta^{(m)}$.  More precisely: to each $f\in \Delta^{(m)}_i$, there corresponds an orbit in $\CPF^{(m)}_i$ isomorphic to $W/W_{\underline{f}}$ as a $W$-set.  So, we have
    \begin{align}
        \tilde \bchi(\CPF^{(m)}) 
            &=
        \sum_{f\in \Delta^{(m)}} (-1)^{\dim(f)} \ch(W/W_{\underline{f}}),\\
        \tilde \bchi(\CPF^{(m),+}) 
            &=
        \sum_{f\in \Delta^{(m),+}} (-1)^{\dim(f)} \ch(W/W_{\underline{f}}).
    \end{align}
    The result is then a direct consequence of Proposition~\ref{prop:altersumcharac}.
\end{proof}



\subsection{Combinatorial models and examples}
\label{sec:combmod}

Let us sketch the combinatorics of $\CPF^{(m)}$ in type $A$, in order to recover the labeled dissections that we defined in the introduction.  Let $f \in \Delta_n^{(m)}$, viewed as a dissection of an $(mn+2)$-gon.  A parking function having $\underline{f}$ as underlying noncrossing partition is obtained by labeling the blocks of $\underline{f}$, as described at the beginning of Section~\ref{sec:defpark}.  Via the bijection mentioned after Lemma~\ref{lemm:bijtypeA}, we can label the inner polygons defined by the dissection $f$, rather than the blocks of $\underline{f}$.  We thus recover the definition of $\CPF_n^{(m)}$ in the introduction.  By properties of the bijection, each inner $(mk+2)$-gon receives $k$ labels.  To complete the picture, it remains to see why the partial orders agree: when we remove a diagonal in $f$ to get $f'$, its two adjacent inner polygons merge and correspond to the blocks of $\underline{f}$ that merge to give $\underline{f'}$.  We omit details.  

Furthermore, it is possible to describe cluster parking functions of type $B$ as centrally symmetric dissections of a $(2n)$-gon (which are the faces of the type $B_n$ cluster complex) and labelling the inner polygons with blocks of a type $B$ set partition, in such a way that the central symmetry sends each block to the opposite block.  More explicitly, consider the complex $\CPF^{(m)}_{2n}$ where the $2n$ labels are now the integers $\{\pm i \;|\; 1\leq i \leq n\}$.  A face of this complex is in $\CPF^{(m)}(B_n)$ iff it is fixed by the half-turn symmetry composed with the relabelling $i\mapsto -i$.  For example, two $1$-dimensional faces of $\CPF(B_2)$ are:
\[
    \begin{tikzpicture}
            \draw[very thick] (0,1) -- (0.866, 0.5) -- (0.866,-0.5) -- (0,-1) -- (-0.866,-0.5) -- (-0.866,0.5) -- (0,1);
            \draw[very thick] (0,1) -- (0,-1);
            \node at (-0.46,0){$12$};
            \node at (0.46,0){$\bar 1 \bar 2$};
    \end{tikzpicture},
    \qquad 
    \begin{tikzpicture}
            \draw[very thick] (0,1) -- (0.866, 0.5) -- (0.866,-0.5) -- (0,-1) -- (-0.866,-0.5) -- (-0.866,0.5) -- (0,1);
            \draw[very thick] (0, 1) --(-0.866,-0.5);
            \draw[very thick] (0.866, 0.5) -- (0,-1);
            \node at (-0.66,0.34){$1$};
            \node at (0,0){$2\bar 2$};
            \node at (0.66,-0.34){$\bar 1$};
    \end{tikzpicture}
\]
(where $\bar i$ means $-i$), and other $1$-dimensional faces are obtained via symmetries of the hexagon and relabelling through the action of signed permutations. 

There is also an explicit combinatorial model in the dihedral case.  Building on the embedding $\Delta^{(m)}(I_2(k)) \subset \Delta^{(m)}(\mathfrak{S}_k)$, there is an embedding $\CPF^{(m)}(I_2(k)) \subset \CPF^{(m)}(\mathfrak{S}_k)$.  Consider a facet in $\CPF^{(m)}(\mathfrak{S}_k)$, such that the underlying $(m+2)$-angulation is in $\Delta^{(m)}(I_2(k))$.  The condition on the labels for this to be a facet in $\CPF^{(m)}(I_2(k))$ is the following: by reading labels from one extremity to the other in the adjacency graph, the sequence has the form $i,i-1,i+1,i-2,i+2,\dots$ or $i,i+1,i-1,i+2,i-2,\dots$ (where $1\leq i \leq k$ and labels are taken modulo $k$).  For example, a facet and its two adjacent vertices in $\CPF^{(3)}(I_2(5))$ are:
\[
    \begin{tikzpicture}[scale=1.4]
        \draw[very thick] (1.00, 0.000) -- (0.932, 0.361) -- (0.739, 0.674) -- (0.446, 0.895) -- (0.0923, 0.996) -- (-0.274, 0.962) -- (-0.603, 0.798) -- (-0.850, 0.526) -- (-0.983, 0.184) -- (-0.983, -0.184) -- (-0.850, -0.526) -- (-0.603, -0.798) -- (-0.274, -0.962) -- (0.0923, -0.996) -- (0.446, -0.895) -- (0.739, -0.674) -- (0.932, -0.361) -- (1.00,0.00);
        \draw[very thick,rounded corners=1] (1.00, 0.000) -- (0.0923, 0.996);
        \draw[very thick,rounded corners=1] (0.932, -0.361) -- (-0.603, 0.798);
        \draw[very thick,rounded corners=1] (-0.850, 0.526) -- (0.446, -0.895);
        \draw[very thick,rounded corners=1] 
        (0.0923, -0.996) -- (-0.983, -0.184);
        \node at (-0.56,-0.68){$1$};
        \node at (-0.38,-0.38){$2$};
        \node at (-0,0){$5$};
        \node at (0.38,0.38){$3$};
        \node at (0.56,0.66){$4$};
    \end{tikzpicture},  \qquad
     \begin{tikzpicture}[scale=1.4]
        \draw[very thick] (1.00, 0.000) -- (0.932, 0.361) -- (0.739, 0.674) -- (0.446, 0.895) -- (0.0923, 0.996) -- (-0.274, 0.962) -- (-0.603, 0.798) -- (-0.850, 0.526) -- (-0.983, 0.184) -- (-0.983, -0.184) -- (-0.850, -0.526) -- (-0.603, -0.798) -- (-0.274, -0.962) -- (0.0923, -0.996) -- (0.446, -0.895) -- (0.739, -0.674) -- (0.932, -0.361) -- (1.00,0.00);
        \draw[very thick,rounded corners=1] (0.932, -0.361) -- (-0.603, 0.798);
        \draw[very thick,rounded corners=1] 
        (0.0923, -0.996) -- (-0.983, -0.184);
        \node at (-0.56,-0.68){$1$};
        \node at (-0.19,-0.19){$25$};
        \node at (0.4,0.52){$34$};
    \end{tikzpicture},  \qquad
     \begin{tikzpicture}[scale=1.4]
        \draw[very thick] (1.00, 0.000) -- (0.932, 0.361) -- (0.739, 0.674) -- (0.446, 0.895) -- (0.0923, 0.996) -- (-0.274, 0.962) -- (-0.603, 0.798) -- (-0.850, 0.526) -- (-0.983, 0.184) -- (-0.983, -0.184) -- (-0.850, -0.526) -- (-0.603, -0.798) -- (-0.274, -0.962) -- (0.0923, -0.996) -- (0.446, -0.895) -- (0.739, -0.674) -- (0.932, -0.361) -- (1.00,0.00);
        \draw[very thick,rounded corners=1] (1.00, 0.000) -- (0.0923, 0.996);
        \draw[very thick,rounded corners=1] (-0.850, 0.526) -- (0.446, -0.895);
        \node at (-0.47,-0.47){$12$};
        \node at (0.19,0.19){$53$};
        \node at (0.56,0.66){$4$};
    \end{tikzpicture}.
\]  
The group $I_2(k)$ acts via the relabelling $i\mapsto i+j$ and $i\mapsto -i+j$ (where indices are taken modulo $k$, and $0\leq j <k$).

We end this section by explicit examples when $m=1$ in type $A_1$ and $A_2$, using the combinatorial description as in the introduction.  Counting faces will illustrate Proposition~\ref{prop:bchicpf}. The $0$-dimensional complex $\CPF_2$ contains 4 isolated vertices, namely:
\[
    \begin{tikzpicture}[scale=0.4]
            \draw[very thick] (0,0) -- (0,2) -- (2,2) -- (2,0) -- (0,0);
            \draw[very thick] (0,2) -- (2,0);
            \node at (0.6,0.6){1};
            \node at (1.4,1.4){2};
    \end{tikzpicture}\,,
    \qquad 
    \begin{tikzpicture}[scale=0.4]
            \draw[very thick] (0,0) -- (0,2) -- (2,2) -- (2,0) -- (0,0);
            \draw[very thick] (0,2) -- (2,0);
            \node at (0.6,0.6){2};
            \node at (1.4,1.4){1};
    \end{tikzpicture}\,,
    \qquad 
    \begin{tikzpicture}[scale=0.4]
            \draw[very thick] (0,0) -- (0,2) -- (2,2) -- (2,0) -- (0,0);
            \draw[very thick] (2,2) -- (0,0);
            \node at (0.6,1.4){2};
            \node at (1.4,0.6){1};
    \end{tikzpicture}\,,
    \qquad 
    \begin{tikzpicture}[scale=0.4]
            \draw[very thick] (0,0) -- (0,2) -- (2,2) -- (2,0) -- (0,0);
            \draw[very thick] (2,2) -- (0,0);
            \node at (1.4,0.6){2};
            \node at (0.6,1.4){1};
    \end{tikzpicture}\,.
\]
So $\tilde\chi(\CPF_2)=-1+4=3 = \#\PF_2$.

The $1$-dimensional complex $\CPF_3$ is the graph in Figure~\ref{fig:cpf3}.  Its number of facets (edges) is $3! \times C_3 = 6\times 5 = 30$.  By Proposition~\ref{prop:cpflink} and the previous example about $\CPF_2$, the graph $\CPF_3$ is $4$-regular (the link of a vertex has Coxeter type $A_1$, which is the only possibility for a rank 1 subtype of $A_2$).  So $\CPF_3$ has $15$ vertices (half the number of edges).  So $\tilde\chi(\CPF_3)=-1+15-30=-16=-\#\PF_3$.  It turns out that this graph can be drawn on a torus: see Figure~\ref{fig:cpf3}, where the red parallelogram encloses a fundamental domain (for readability, we did not write the labeled triangulations corresponding to each edge).  The sides of this parallelogram are the vectors $(3,-3)$ and $(4,1)$.  If $\sigma = (23)$, it is easily seen that the fixed-point subcomplex $\CPF_3^{\sigma}$ contains 5 isolated vertices, namely those where a triangle is labeled 1 and a quadrangle is labeled 23.  So, $\tilde\chi(\CPF_3^{\sigma})=-1+5 = 4$ (which is also the number of elements in $\PF_3$ fixed by $\sigma$).

\begin{figure}
    \centering
        \begin{tikzpicture}[scale=0.3]
            \usetikzlibrary {arrows.meta}
            \begin{scope}[xshift=-7cm,yshift=7cm]         \draw[very thick,color=red] (0,0) -- (21,-21) -- (49,-14) -- (28,7) -- (0,0);
            \draw[very thick,color=red,->,arrows = {-Stealth[scale=1.5]}] (0,0) -- (10.5,-10.5);
            \draw[very thick,color=red,->,arrows = {-Stealth[scale=1.5]}] (28,7) -- (38.5,-3.5);
            \draw[very thick,color=red,->,arrows = {-Triangle[scale=1.5,fill=none]}] (0,0) -- (14,3.5);
            \draw[very thick,color=red,->,arrows = {-Triangle[scale=1.5,fill=none]}] (21,-21) -- (35,-17.5);
            \end{scope}
            \draw[very thick] (0,2) -- (1.9,0.62) -- (1.18,-1.62) -- (-1.18,-1.62) -- (-1.9,0.62) -- (0,2);
            \draw[very thick] (0,2) -- (-1.18,-1.62);
            \node at (-1.2,0.3){1};
            \node at (0.5,-0.4){23};
            \begin{scope}[xshift=7cm,yshift=-7cm]
                \draw[very thick] (0,2) -- (1.9,0.62) -- (1.18,-1.62) -- (-1.18,-1.62) -- (-1.9,0.62) -- (0,2);
                \draw[very thick] (0,2) -- (-1.18,-1.62);
                \node at (-1.2,0.3){2};
                \node at (0.5,-0.4){13};
            \end{scope}
            \begin{scope}[xshift=-7cm,yshift=7cm]
                \draw[very thick] (0,2) -- (1.9,0.62) -- (1.18,-1.62) -- (-1.18,-1.62) -- (-1.9,0.62) -- (0,2);
                \draw[very thick] (0,2) -- (-1.18,-1.62);
                \node at (-1.2,0.3){3};
                \node at (0.5,-0.4){12};
            \end{scope}
            \begin{scope}[xshift=0cm,yshift=7cm]
                \draw[very thick] (0,2) -- (1.9,0.62) -- (1.18,-1.62) -- (-1.18,-1.62) -- (-1.9,0.62) -- (0,2);
                \draw[very thick] (1.9,0.62) -- (-1.18,-1.62);
                \node at (0.8,-0.9){2};
                \node at (-0.4,0.4){13};
            \end{scope}
            \begin{scope}[xshift=7cm,yshift=0cm]
                \draw[very thick] (0,2) -- (1.9,0.62) -- (1.18,-1.62) -- (-1.18,-1.62) -- (-1.9,0.62) -- (0,2);
                \draw[very thick] (1.9,0.62) -- (-1.18,-1.62);
                \node at (0.8,-0.9){3};
                \node at (-0.4,0.4){12};
            \end{scope}
            \begin{scope}[xshift=14cm,yshift=-7cm]
                \draw[very thick] (0,2) -- (1.9,0.62) -- (1.18,-1.62) -- (-1.18,-1.62) -- (-1.9,0.62) -- (0,2);
                \draw[very thick] (1.9,0.62) -- (-1.18,-1.62);
                \node at (0.8,-0.9){1};
                \node at (-0.4,0.4){23};
            \end{scope}
            \begin{scope}[xshift=7cm,yshift=7cm]
                \draw[very thick] (0,2) -- (1.9,0.62) -- (1.18,-1.62) -- (-1.18,-1.62) -- (-1.9,0.62) -- (0,2);
                \draw[very thick] (-1.9,0.62) -- (1.9,0.62);
                \node at (0,1.2){1};
                \node at (0,-0.4){23};
            \end{scope}
            \begin{scope}[xshift=14cm,yshift=0cm]
                \draw[very thick] (0,2) -- (1.9,0.62) -- (1.18,-1.62) -- (-1.18,-1.62) -- (-1.9,0.62) -- (0,2);
                \draw[very thick] (-1.9,0.62) -- (1.9,0.62);
                \node at (0,1.2){2};
                \node at (0,-0.4){13};
            \end{scope}
            \begin{scope}[xshift=21cm,yshift=-7cm]
                \draw[very thick] (0,2) -- (1.9,0.62) -- (1.18,-1.62) -- (-1.18,-1.62) -- (-1.9,0.62) -- (0,2);
                \draw[very thick] (-1.9,0.62) -- (1.9,0.62);
                \node at (0,1.2){3};
                \node at (0,-0.4){12};
            \end{scope}
            \begin{scope}[xshift=14cm,yshift=7cm]
                \draw[very thick] (0,2) -- (1.9,0.62) -- (1.18,-1.62) -- (-1.18,-1.62) -- (-1.9,0.62) -- (0,2);
                \draw[very thick] (-1.9,0.62) -- (1.18,-1.62);
                \node at (-0.8,-0.9){3};
                \node at (0.4,0.4){12};
            \end{scope}
            \begin{scope}[xshift=21cm,yshift=0cm]
                \draw[very thick] (0,2) -- (1.9,0.62) -- (1.18,-1.62) -- (-1.18,-1.62) -- (-1.9,0.62) -- (0,2);
                \draw[very thick] (-1.9,0.62) -- (1.18,-1.62);
                \node at (-0.8,-0.9){1};
                \node at (0.4,0.4){23};
            \end{scope}
            \begin{scope}[xshift=28cm,yshift=-7cm]
                \draw[very thick] (0,2) -- (1.9,0.62) -- (1.18,-1.62) -- (-1.18,-1.62) -- (-1.9,0.62) -- (0,2);
                \draw[very thick] (-1.9,0.62) -- (1.18,-1.62);
                \node at (-0.8,-0.9){2};
                \node at (0.4,0.4){13};
            \end{scope}
            \begin{scope}[xshift=21cm,yshift=7cm]
                \draw[very thick] (0,2) -- (1.9,0.62) -- (1.18,-1.62) -- (-1.18,-1.62) -- (-1.9,0.62) -- (0,2);
                \draw[very thick] (0,2) -- (1.18,-1.62);
                \node at (1.2,0.3){2};
                \node at (-0.5,-0.4){13};
            \end{scope}
            \begin{scope}[xshift=28cm,yshift=0cm]
                \draw[very thick] (0,2) -- (1.9,0.62) -- (1.18,-1.62) -- (-1.18,-1.62) -- (-1.9,0.62) -- (0,2);
                \draw[very thick] (0,2) -- (1.18,-1.62);
                \node at (1.2,0.3){3};
                \node at (-0.5,-0.4){12};
            \end{scope}
            \begin{scope}[xshift=35cm,yshift=-7cm]
                \draw[very thick] (0,2) -- (1.9,0.62) -- (1.18,-1.62) -- (-1.18,-1.62) -- (-1.9,0.62) -- (0,2);
                \draw[very thick] (0,2) -- (1.18,-1.62);
                \node at (1.2,0.3){1};
                \node at (-0.5,-0.4){23};
            \end{scope}
            \begin{scope}[xshift=0cm,yshift=0cm]
                \draw[very thick] (2,0) -- (5,0);
            \end{scope}
            \begin{scope}[xshift=7cm,yshift=0cm]
                \draw[very thick] (2,0) -- (5,0);
            \end{scope}
            \begin{scope}[xshift=14cm,yshift=0cm]
                \draw[very thick] (2,0) -- (5,0);
            \end{scope}
            \begin{scope}[xshift=21cm,yshift=0cm]
                \draw[very thick] (2,0) -- (5,0);
            \end{scope}
            \begin{scope}[xshift=-7cm,yshift=7cm]
                \draw[very thick] (2,0) -- (5,0);
            \end{scope}
            \begin{scope}[xshift=0cm,yshift=7cm]
                \draw[very thick] (2,0) -- (5,0);
            \end{scope}
            \begin{scope}[xshift=7cm,yshift=7cm]
                \draw[very thick] (2,0) -- (5,0);
            \end{scope}
            \begin{scope}[xshift=14cm,yshift=7cm]
                \draw[very thick] (2,0) -- (5,0);
            \end{scope}
            \begin{scope}[xshift=7cm,yshift=-7cm]
                \draw[very thick] (2,0) -- (5,0);
            \end{scope}
            \begin{scope}[xshift=14cm,yshift=-7cm]
                \draw[very thick] (2,0) -- (5,0);
            \end{scope}
            \begin{scope}[xshift=21cm,yshift=-7cm]
                \draw[very thick] (2,0) -- (5,0);
            \end{scope}
            \begin{scope}[xshift=28cm,yshift=-7cm]
                \draw[very thick] (2,0) -- (5,0);
            \end{scope}
            \begin{scope}[xshift=35cm,yshift=-7cm]
                \draw[very thick] (2,0) -- (5,0);
            \end{scope}
            \begin{scope}[xshift=28cm,yshift=0cm]
                \draw[very thick] (2,0) -- (5,0);
            \end{scope}
            \begin{scope}[xshift=21cm,yshift=7cm]
                \draw[very thick] (2,0) -- (5,0);
            \end{scope}
            \begin{scope}[xshift=0cm,yshift=0cm]
                \draw[very thick] (0,2.2) -- (0,5);
            \end{scope}
            \begin{scope}[xshift=7cm,yshift=0cm]
                \draw[very thick] (0,2.2) -- (0,5);
            \end{scope}
            \begin{scope}[xshift=14cm,yshift=0cm]
                \draw[very thick] (0,2.2) -- (0,5);
            \end{scope}
            \begin{scope}[xshift=21cm,yshift=0cm]
                \draw[very thick] (0,2.2) -- (0,5);
            \end{scope}
            \begin{scope}[xshift=28cm,yshift=0cm]
                \draw[very thick] (0,2.2) -- (0,5);
            \end{scope}
            \begin{scope}[xshift=7cm,yshift=-7cm]
                \draw[very thick] (0,2.2) -- (0,5);
            \end{scope}
            \begin{scope}[xshift=14cm,yshift=-7cm]
                \draw[very thick] (0,2.2) -- (0,5);
            \end{scope}
            \begin{scope}[xshift=21cm,yshift=-7cm]
                \draw[very thick] (0,2.2) -- (0,5);
            \end{scope}
            \begin{scope}[xshift=28cm,yshift=-7cm]
                \draw[very thick] (0,2.2) -- (0,5);
            \end{scope}
            \begin{scope}[xshift=14cm,yshift=-14cm]
                \draw[very thick] (0,2.2) -- (0,5);
            \end{scope}
            \begin{scope}[xshift=21cm,yshift=7cm]
                \draw[very thick] (0,2.2) -- (0,5);
            \end{scope}
            \begin{scope}[xshift=21cm,yshift=-14cm]
                \draw[very thick] (0,2.2) -- (0,5);
            \end{scope}
            \begin{scope}[xshift=14cm,yshift=7cm]
                \draw[very thick] (0,2.2) -- (0,5);
            \end{scope}
            \begin{scope}[xshift=35cm,yshift=-7cm]
                \draw[very thick] (0,2.2) -- (0,5);
            \end{scope}
            \begin{scope}[xshift=28cm,yshift=-14cm]
                \draw[very thick] (0,3.5) -- (0,5);
            \end{scope}
            \begin{scope}[xshift=7cm,yshift=7cm]
                \draw[very thick] (0,2.2) -- (0,3.5);
            \end{scope}
            \begin{scope}[xshift=35cm,yshift=0cm,color=gray]
                \draw[very thick] (0,2) -- (1.9,0.62) -- (1.18,-1.62) -- (-1.18,-1.62) -- (-1.9,0.62) -- (0,2);
                \draw[very thick] (0,2) -- (-1.18,-1.62);
                \node at (-1.2,0.3){2};
                \node at (0.5,-0.4){13};
            \end{scope}
            \begin{scope}[xshift=28cm,yshift=7cm,color=gray]
                \draw[very thick] (0,2) -- (1.9,0.62) -- (1.18,-1.62) -- (-1.18,-1.62) -- (-1.9,0.62) -- (0,2);
                \draw[very thick] (0,2) -- (-1.18,-1.62);
                \node at (-1.2,0.3){1};
                \node at (0.5,-0.4){23};
            \end{scope}
            \begin{scope}[xshift=21cm,yshift=14cm,color=gray]
                \draw[very thick] (0,2) -- (1.9,0.62) -- (1.18,-1.62) -- (-1.18,-1.62) -- (-1.9,0.62) -- (0,2);
                \draw[very thick] (0,2) -- (-1.18,-1.62);
                \node at (-1.2,0.3){3};
                \node at (0.5,-0.4){12};
            \end{scope}
            \begin{scope}[xshift=42cm,yshift=-7cm,color=gray]
                \draw[very thick] (0,2) -- (1.9,0.62) -- (1.18,-1.62) -- (-1.18,-1.62) -- (-1.9,0.62) -- (0,2);
                \draw[very thick] (0,2) -- (-1.18,-1.62);
                \node at (-1.2,0.3){3};
                \node at (0.5,-0.4){12};
            \end{scope}
            \begin{scope}[xshift=14cm,yshift=-14cm,color=gray]
                \draw[very thick] (0,2) -- (1.9,0.62) -- (1.18,-1.62) -- (-1.18,-1.62) -- (-1.9,0.62) -- (0,2);
                \draw[very thick] (0,2) -- (-1.18,-1.62);
                \node at (-1.2,0.3){3};
                \node at (0.5,-0.4){12};
            \end{scope}
        \end{tikzpicture}
    \caption{The graph $\CPF_3$, drawn on a torus.\label{fig:cpf3}}
\end{figure}


One can check that $\CPF(I_2(m))$ is a $4$-regular graph that can be drawn on a torus, just as in the case of $\CPF_3$.  Here, the sides of the fundamental parallelogram for this torus are the vectors $(m,-m)$ and $(m+1,1)$.


\section{Topology on noncrossing partitions}
\label{sec:topnc}

The {\it order ideal} generated by a subset $X\subset P$ in a poset $P$ is:
\[
    \mathcal{I}(P,X)
    :=
    \big\{ 
        y \in P 
            \;:\; 
        \exists x \in X, \; y\leq x
    \big\}.
\]
It is considered as a poset itself, where the order is inherited from $P$.  The goal of this section is to show that certain order ideals in $\bNC$ (the proper part of $\NC$) have the homotopy type of a wedge of spheres, and this will be used in the next section.

Define the {\it inversion set} of $w\in W$ by:
\[
    \Inv(w)
        :=
    \big\{
         t \in T \; : \; \ell_S(wt) < \ell_S(w)
    \big\}
\]
where $\ell_S$ is the Coxeter length (the reflection length $\ell$ could be denoted by $\ell_T$).  
It is well-known that inversions can be reformulated in terms of roots as follows:
\[
    \Inv(w) 
    =
    \big\{
        t_\alpha \;:\;
        \alpha \in \Phi_+ \cap w^{-1}(\Phi_-)
    \big\}.
\]

\begin{prop} \label{prop:idealNC}
  Let $w\in W$.  Assume $w\neq e$ (equivalently, $\Inv(w)\neq \varnothing$).  Then, 
  we have:
    \[
        \Omega\big( \mathcal{I}(\bNC , \Inv(w) ) \big)
            \simeq
        (\mathbb{S}^{n-2})^{\vee k_w},
    \]
    where the integer $k_w$ is:
    \begin{align} \label{def:kw}
        k_w :=
        \# \big\{
            f \in \Delta^+_{n-1} \;:\; 
            f \subset \Inv(w)
        \big\}.
    \end{align}
\end{prop}

If $w$ is the longest element of $W$, denoted by $w_\circ$, we have $\Inv(w_\circ)=T$ and the order ideal in the proposition is the whole poset $\bNC$.  Athanasiadis {\it et al.}~\cite{athanasiadistzanaki} have proved via EL-shellability that $\Omega(\bNC) \simeq (\mathcal{S}^2)^{\vee (\Cat^+)}$, where $\Cat^+ = \# \Delta^+_{n-1}$ is the number of facets in $\Delta^+$.  Our statement is a generalization, and the proof is indeed inspired from~\cite{athanasiadistzanaki} (and will take the rest of this section).

The proof heavily relies on the combinatorics and geometry of the positive cluster complex $\Delta^+$ (which is $\Delta^{(1),+}(W)$ in earlier notation).  A key property is the convexity of an inversion set, and to use it we need to describe the geometry of $\Delta^+$ rather than seeing it as an abstract simplicial complex.   Consider a vector $\rho$ such that $\langle \alpha | \rho \rangle >0$ for every positive root $\alpha$, and define
\begin{align*}
    H &:= \big\{ v\in V \;:\;  \langle v|\rho\rangle=0 \big\},\quad
    H^1 := \big\{ v\in V \;:\;  \langle v|\rho\rangle=1 \big\}, \\
    H^+ &:= \big\{ v\in V \;:\; \langle v|\rho\rangle>0 \big\},\quad
    H^- := \big\{ v\in V \;:\; \langle v|\rho\rangle<0 \big\}.
\end{align*}
From the choice of $\rho$, it is easily seen that 
$\Span_{\mathbb{R}^+}(\Pi) \cap H^1$ is an $(n-1)$-dimensional simplex in $H^1$, and its vertices are the simple roots (up to a positive scalar factor). 

\begin{lemm}
    The simplices $\Span_{\mathbb{R}^+}(f) \cap H^1$ for $f \in \Delta^+$ form a triangulation of the simplex $\Span_{\mathbb{R}^+}(\Pi) \cap H^1$, and this triangulation is combinatorially isomorphic to $\Delta^+$.
\end{lemm}

\begin{proof}
The cones $\Span_{\mathbb{R}^+}(f)$ form a simplicial fan, as a convex subset in the {\it cluster fan} (see~\cite{fominzelevinsky}).  Taking the intersection with the affine hyperplane thus gives a triangulation. 
\end{proof}

See for example the left part in Figure~\ref{fig:theta}, in the case $A_3$ and a bipartite Coxeter element (where $(i,j)$ is a reflection in the symmetric group $\mathfrak{S}_4$, {\it i.e.}, a transposition).  We can thus see the geometric realization of $\Delta^+$ (or any of its faces) as a simplex in $H^1$. 

Define a subposet $\Theta(w)$ of (the face poset of) $\Delta^+$ as:
\[
    \Theta(w) 
        :=
    \big\{
        f \in \Delta^+ \; : \;
        f \cap \Inv(w) \neq \varnothing
    \big\},
\]
and a subposet $\Theta'(w)$ of $\Theta(w)$ by removing elements of $\Delta^+_{n-1}$:
\[
    \Theta'(w) 
        :=
    \big\{
        f \in \Theta(w) \; : \; 
        \dim(f)<n-1
    \big\}.
\]
In particular, if $\Theta(w) \cap \Delta^+_{n-1} = \varnothing$ then $\Theta(w)=\Theta'(w)$.

For example, see the right part of Figure~\ref{fig:theta}.  Here, $w$ is the permutation $\sigma\in\mathfrak{S}_4$ with inversion set $\{(1,4),(2,3),(2,4),(3,4)\}$ (namely, $\sigma = 2431$).  The hyperplane $w^{-1}(H)$ is the red line, it separates the inversions of $w$ (to its right) from its non-inversions (to its left).  The poset $\Theta(w)$ contains the faces having a vertex in $H^-$ (the half-plane to the right of the red line), this gives $4$ vertices, $9$ edges, and all $5$ clusters.  The poset $\Theta'(w)$ contains $4$ vertices and $9$ edges (thick edges in the picture).  Note that these posets don't contain the empty face. 

\begin{figure}
    \centering
    \begin{tikzpicture}[scale=2]
        \draw[fill=lightgray] (0,0) -- (1,1.732) -- (2,0) -- (0,0);
        \draw (0,0) -- (1.5,0.866);
        \draw (2,0) -- (0.5,0.866);
        \draw (1.5,0.866) -- (0.5,0.866);
        \node at (-0.2,-0.2) {$(1,2)$};
        \node at (2.2,-0.2) {$(3,4)$};
        \node at (1,1.9) {$(2,3)$};
        \node at (1,0.3) {$(1,4)$};
        \node at (0.2,1) {$(1,3)$};
        \node at (1.8,1) {$(2,4)$};
    \end{tikzpicture}
    \qquad 
        \begin{tikzpicture}[scale=2]
        \draw[fill=lightgray] (0,0) -- (1,1.732) -- (2,0) -- (0,0);
        \draw (0,0) -- (1.5,0.866);
        \draw (2,0) -- (0.5,0.866);
        \draw (1.5,0.866) -- (0.5,0.866);
        \node at (-0.2,-0.2) {\phantom{$(1,2)$}};
        \node at (2.2,-0.2) {\phantom{$(3,4)$}};
        \draw[color=red] (0.7,2.3) -- (0.7,-0.3);
        \node[color=red] at (0.1,2.1) {$w^{-1}(H_+)$};
        \node[color=red] at (1.3,2.1) {$w^{-1}(H_-)$};
        \draw[line width=1mm,line cap=round] (1,1.732) -- (2,0);
        \draw[line width=1mm,line cap=round] (2,0) -- (1,0.58) -- (1.5,0.866);
        \draw[line width=1mm,line cap=round, loosely dashed] (1,1.732) -- (0.7,1.21);
        \draw[line width=1mm,line cap=round, loosely dashed] (1.5,0.866) -- (0.7,0.866);
        \draw[line width=1mm,line cap=round, loosely dashed] (2,0) -- (0.7,0);
        \draw[line width=1mm,line cap=round, loosely dashed] (1,0.58) -- (0.7,0.41);
        \draw[line width=1mm,line cap=round, loosely dashed] (1,0.58) -- (0.7,0.75);
    \end{tikzpicture}
    \caption{The complex $\Delta^+$ in type $A_2$ with a bipartite Coxeter element (left). Illustration of the poset $\Theta'(w)$ (right).}
    \label{fig:theta}
\end{figure}

Though these posets are what we really use in the proof, to write precise arguments it is convenient to use the following two simplicial complexes:
\[
    \Xi(w)
        :=
    \big\{
        f \in \Delta^+ \; : \; 
        f \subset \Inv(w)
    \big\}
\]
is the full subcomplex of $\Delta^+$ with vertex set $\Inv(w)$, and:
\[
    \Xi'(w)
        :=
    \big\{
        f \in \Xi(w) \; : \; 
        \dim(f) < n-1
    \big\}
\]
is its $(n-2)$-skeleton.  Note that the integer $k_w$ defined in~\eqref{def:kw} is the number of $(n-1)$-dimensional faces of $\Xi(w)$.  The complex $\Xi'(w)$ has 4 vertices and 4 edges (the four plain thick edges in Figure~\ref{fig:theta}), and the poset $\Theta'(w)$ has 5 more elements (the dashed edges). The complex $\Xi(w)$ contains one cluster (the one entirely to the right of the red line), and there are 4 clusters that are in $\Theta(w)$ but not in $\Xi(w)$ (those through which goes the red line).  

Geometrically, we can see $\Theta(w)$ as the {\it open star} of $\Xi(w)$ inside $\Delta^+$ (the open star is usually defined only for a face rather than a subcomplex, see~\cite{kozlov}).  Similarly, $\Theta'(w)$ is the open star of $\Xi'(w)$ inside the $(n-2)$-skeleton of $\Delta^+$.

\begin{lemm} \label{lemm:thetaxi}
    We have $\Omega(\Theta(w))\simeq \Xi(w)$, and $\Omega(\Theta'(w))\simeq \Xi'(w)$.
\end{lemm}

\begin{proof}
Any subcomplex of a complex is a deformation retraction of its open star (under mild conditions), and in particular a subcomplex and its open star are homotopy equivalent.  However, we did not formally defined the geometric realization of open stars, so let us also give a more formal poset-theoretic proof.  See~\cite[Lemma~70.1]{munkres} for the relevant statement in terms of simplicial complexes.

Suppose that we have two posets $P_1$ and $P_2$, and two poset maps $f,g:P_1\to P_2$ such that $f(p)\leq g(p)$ for all $p\in P_1$.  It was shown by Quillen~\cite[Section~1.3]{quillen} that the maps $\Omega(P_1)\to \Omega(p_2)$ induced by $f$ and $g$ are then homotopy equivalent.  Here, we consider the map $\phi:\Theta(w) \to \Xi(w)$ defined by 
\[
    f \mapsto f \cap \Inv(w).
\]
The image of $\phi$ is the poset of nonempty faces of $\Xi(w)$.  We also consider the map $\psi$ defined as the inclusion of $\operatorname{im}(\phi)$ in $\Theta(w)$.  We have $(\phi\circ\psi)(p)=p$ for all $p\in\operatorname{im}(\phi)$, and $(\psi\circ\phi)(p) \leq p$ for all $p\in\Theta(w)$.  By the result of Quillen mentioned above, the map induced by $\phi$ and $\psi$ on the order complexes are homotopy equivalences inverse to each other.  This proves the result.  The statement about $\Theta'(w)$ and $\Xi'(w)$ is similar.
\end{proof}

\begin{lemm} \label{lemm:thetaxitoptriv}
    The complexes $\Omega(\Theta(w))$ and $\Xi(w)$ are topologically trivial.
\end{lemm}

\begin{proof}
Consider $X := w^{-1}(H^-) \cap |\Delta^+| \subset H^1 $.  In the example of Figure~\ref{fig:theta}, it is the part of the gray triangle that lies to the right of the red line.  It is convex as an intersection of two convex subsets in $H^1$ (a half-space and a simplex), so it is contractible.  The geometric realization of a vertex $t\in \Delta^+_1$ (as a simplex in $H^1$) lies in $X$ if and only if $t\in \Inv(w)$, by the characterization of inversions in terms of roots.

We build a deformation retraction from $X$ to $|\Xi(w)|$ as follows.  Let $v \in X$, and $f \in \Theta(w)$ be such that $v\in |f|$.  There is a unique decomposition $v = (1-\lambda) v^+ + \lambda v^-$ where $v^+ \in |f  \cap w^{-1}(H^+) |$, $v^- \in |f \cap w^{-1}(H^-) | $, and $\lambda \in [0,1]$.  More concretely, we write $v$ as a convex combination of the vertices of $f$, then keep the vertices that correspond to inversions of $w$ (respectively, non-inversions of $w$) to get $v_-$ (respectively, $v_+$).  This is done with the convention that if $v \in |\Xi(w)|$, $v^+$ doesn't exist and $v=v^-$.  

Now, the deformation retraction $H : [0,1]\times X \to |\Xi(w)|$ is defined as follows:
\[
    H(t,v) := (1-t) v  + t v^-.
\]
The map $v \mapsto v^-$ is affine on each simplex $|f|$ for $f\in\Theta(w)$, and in particular continuous.  It follows that $H$ is affine (hence continuous) on each subspace $[0,1]\times |f|$.  Clearly, we have $H(0,v)=v$, $H(1,v) = v^- \in |\Xi(w)|$, and $H(t,v)=v$ if $v \in |\Xi(w)|$.  As a deformation retraction is a homotopy equivalence, $\Xi(w)$ is topologically trivial.  By the  previous lemma, $\Omega(\Theta(w))$ is as well.
\end{proof}

\begin{lemm}
    We have $\Omega(\Theta'(w)) \simeq (\mathbb{S}^{n-2})^{\vee k_w}$.
\end{lemm}

\begin{proof}
    We prove $\Xi'(w) \simeq (\mathbb{S}^{n-2})^{\vee k_w}$, and the result will follow by Lemma~\ref{lemm:thetaxi}.  The idea is to note that $\Xi(w)$ is topologically trivial (by the previous lemma) and $\Xi'(w)$ is obtained by removing $k_w$ $(n-1)$-dimensional facets.  
    
    It is straightforward to describe the chain complex of $\Xi'(w)$ in terms of that of $\Xi(w)$, and we get that $\tilde H_{n-2}(\Xi'(w))$ is isomorphic to $\mathbb{Z}^{k_w}$ and all other homology groups vanish.  If $n>3$, $\Xi'(w)$ is simply-connected because the fundamental group of a complex only depends on its $2$-skeleton (and $\Xi(w)$ is simply-connected).  We can use Proposition~\ref{prop:characwedge} to complete the proof.
\end{proof}



To prove the last lemma of this series, we use the complex $\Delta^+(\pi)$ for $\pi\in\NC$ (the case $m=1$ of the complex $\Delta^{(m),+}(\pi)$ defined in Proposition~\ref{prop:subcompl}).  When $\ell(\pi)>0$, the geometric realization of $\Delta^+(\pi)$ is a $(\ell(\pi)-1)$-dimensional ball and in particular this complex is topologically trivial.  See~\cite[Corollary~7.7]{bradywatt}.  Another way to see this is to identify $\Delta^+(\pi)$ with the positive cluster complex associated to the parabolic subgroup $W_{\pi}$ (but a technical point is that we need a definition of $\Delta^+$ for each standard Coxeter element, not just the bipartite Coxeter element, see for example~\cite{bianejosuatverges}).

\begin{lemm}
    We have:
    \[
        \Omega(\Theta'(w))
            \simeq
        \Omega(\mathcal{I}(\bNC , \Inv(w) )).
    \]
\end{lemm}

\begin{proof}
Define a poset map
\[
    \rho : \Theta'(w) \to \mathcal{I}(\bNC , \Inv(w) )
\]
by $\rho(f) = \prod f$.  Our goal is to apply Quillen fiber lemma~\cite[Proposition~1.6]{quillen} to get the result.

Let us first check that the image of $\rho$ is indeed in $\mathcal{I}(\bNC , \Inv(w) )$.  First note that by definition of $\Theta'(w)$, its elements have cardinality between $1$ and $n-1$, so that $f\in \Theta'(w)$ implies $ \prod f\in  \bNC$.  Next, $f\in \Theta'(w)$ means that $f \in \Delta^+$ contains a vertex $t \in \Inv(w)$. So, $\prod f \leq t $ and it follows $\prod f \in \mathcal{I}(\bNC , \Inv(w) )$.  

Now, for $\pi \in \mathcal{I}(\bNC , \Inv(w) )$, consider its {\it fiber}:
\[
    \pi\backslash\rho
    :=
    \big\{ 
        f \in \Theta'(w) 
            \;: \; 
        \prod f \geq \pi \big\}.
\]
It is a subposet of $\Delta^+(\pi^{\kappa})$, because the condition $\prod f \geq \pi$ is equivalent to $\underline{f} \leq \pi^{\kappa}$.  By definition of $\Theta'(w)$, we also have:
\begin{align} \label{eq:fiberpirho}
    \pi\backslash\rho
        =
    \big\{
        f \in \Delta^+(\pi^{\kappa})
        \; : \;
        f \cap \Inv(w) \neq \varnothing \text{ and }
        \dim(f)<n-1
    \big\}.
\end{align}
Note that the condition $\dim(f)<n-1$ can be removed: if $\dim(f)=n-1$ for $f\in \Delta^+$, we have $\prod f = c \not\leq \pi^{\kappa}$ (since $\pi\in \bNC$).  From~\eqref{eq:fiberpirho}, we can identify the fiber $\pi\backslash\rho$ with a poset $\Theta'(v)$, where:
\begin{itemize}
    \item Instead of $\Delta^+$, the poset is defined relatively to $\Delta^+(\pi^{\kappa})$ (the complex $\Delta^+$ associated to the parabolic subgroup $W_{\pi^\kappa}$),
    \item $v \in W_{\pi^\kappa}$ is defined by the condition $\Inv(v) = \Inv(w) \cap W_{\pi}$.
\end{itemize}
It follows from Lemma~\ref{lemm:thetaxitoptriv} that $\pi\backslash \rho \simeq \bullet$.  

Via Quillen's fiber lemma~\cite[Proposition~1.6]{quillen}, $\rho$ induces an homotopy equivalence of the order complexes of $\Theta'(w)$ and $\mathcal{I}(\bNC , \Inv(w) )$. 
\end{proof}

\begin{proof}[Proof of Proposition~\ref{prop:idealNC}]
This is an immediate consequence of the previous two lemmas. 
\end{proof}


\section{Topology of the parking function poset}
\label{sec:toppf}

In this section, we describe the homotopy type of $\Omega(\bPF)$ as a wedge of spheres.  Note that the type $A$ case has been obtained in~\cite[Section~2]{delcroixogerjosuatvergesrandazzo}, with a rather intricate method leading to shellability of the poset.

We use here a mixture of different methods (and we don't prove that $\bPF$ is shellable).  To begin, the next two lemmas give an inductive tool to show that a complex has the homotopy type of a wedge of spheres.  

\begin{lemm}
  Let $d\geq 1$.  Suppose that a $d$-dimensional simplicial complex $X$ is a union of two $d$-dimensional subcomplexes $A \cup B$, where for some $k,\ell \geq 0$ we have:
  \begin{itemize}
      \item $A \simeq (\mathbb{S}^d)^{\vee k}$,
      \item $B$ is topologically trivial,
      \item $A\cap B \simeq (\mathbb{S}^{d-1})^{\vee \ell}$ (where the intersection is taken in the geometric realization of $X$).
  \end{itemize}
  Then we have $X \simeq (\mathbb{S}^d)^{\vee(k+\ell)}$.
\end{lemm}

\begin{proof}
We use the Mayer-Vietoris long exact sequence in homology (in the framework of simplicial complexes, see~\cite{kozlov}).  It gives that $\tilde H_d(X)$ is an extension of $\tilde H_d(A) \oplus \tilde H_d(B) = \tilde  H_d(A) \simeq \mathbb{Z}^k$ by $\tilde H_{d-1}(A\cap B) \simeq \mathbb{Z}^\ell$.  Therefore, there is an isomorphism $\tilde H_d(X) \simeq \mathbb{Z}^{k+\ell}$.  Moreover, the other reduced homology groups vanish.

If $d=1$, the result is clear because any connected $1$-dimensional complex is a wedge of $\mathbb{S}^1$.  Now assume $d \geq 2$, then a straightforward application of the (simplicial) Van Kampen theorem shows that $X$ is simply connected.  Since moreover $X$ has a unique nonzero homology group which is a free abelian group, this determines its homotopy type as being a wedge of spheres via Proposition~\ref{prop:characwedge}.
\end{proof}


\begin{lemm} \label{lemm:genshell}
    Suppose that a $d$-dimensional complex $X$ is a union of $d$-dimensional subcomplexes $A_1\cup A_2 \cup \dots \cup A_m$, where:
    \begin{itemize}
        \item $\forall i \in \llbracket 1,m\rrbracket$, $A_i$ is topologically trivial,
        \item $\forall i \in \llbracket 2,m\rrbracket$, $\exists k_i \geq 0$, such that $\big(\cup_{j=1}^{i-1} A_j \big) \cap A_i \simeq (\mathbb{S}^{d-1})^{\vee k_i}$. 
    \end{itemize}
    Then $X \simeq (\mathbb{S}^d)^{\vee k}$ with $k=\sum_{i=2}^{m} k_i$.
\end{lemm}

\begin{proof}
    By iterating the previous lemma, an immediate induction on $i$ shows that
    \[
        \bigcup_{j=1}^i A_j
            \simeq
        (\mathbb{S}^d)^{\vee \ell_i}
    \]
    where $\ell_i = \sum_{j=2}^i k_j$.  We then take $i=m$.
\end{proof}

Note that the previous lemma encompasses the notion of shellability (which, corresponds to the case where each $A_i$ is a $d$-simplex and each $\big(\cup_{j=1}^i A_j \big) \cap A_{i+1}$ a nonempty union of $(d-1)$-simplices).

The general idea is to apply the previous lemma to $\Omega(\bPF)$, by taking each $A_i$ to be of the form $\Omega(\mathcal{I}(\bPF,\{ \mu \}))$ where $\mu$ is a maximal element of $\PF$.  It remains to find an appropriate total order on the maximal elements of $\PF$ (we will see below that we can take any linear extension of the Bruhat order), and understand the topology of the intersections of subcomplexes as in the second assumption of Lemma~\ref{lemm:genshell}.  A couple of lemmas will help dealing with the latter point.

\begin{lemm} \label{lemm:equ}
    Let $ w' W_{\pi'}, w W_{\pi} \in \PF$ such that $w' W_{\pi'} \leq w W_{\pi}$.  Assume that $w'$ and $w$ are minimal length coset representatives.  Then $w'\leq w$ in the Bruhat order on $W$, and
    \[
        w' = w
            \Longleftrightarrow
        \pi' \notin \mathcal{I}(\NC,\Inv(w)).
    \]
\end{lemm}

\begin{proof}
By Lemma~\ref{lemm:pftonc}, the relation $w' W_{\pi'} \leq w W_{\pi}$ implies $\pi' \leq \pi$.  We deduce $W_{\pi} \subset W_{\pi'}$.  

By definition, the relation $w' W_{\pi'} \leq w W_{\pi}$ means that we have the inclusion $w W_{\pi} \subset w' W_{\pi'}$.  Multiplying on the right by $W_{\pi'}$ (and using $W_{\pi} \subset W_{\pi'}$), we get $w W_{\pi'} = w' W_{\pi'}$.  Since $w'$ is a minimal length representative in $W/W_{\pi'}$, we get $w' \leq w$.

The case of equality means that $w$ is also a minimal-length representative in $W / W_{\pi'}$.  This is also equivalent to $\Inv(w) \cap W_{\pi'} = \varnothing$.  Since
\[
    T \cap W_{\pi'} 
        =
    \{t \in T \;:\; \pi' \leq t\},
\]
the condition $\Inv(w) \cap W_{\pi'} = \varnothing$ is also equivalent to 
\[
    \Inv(w) \cap \{t \in T \;:\; \pi' \leq t\}
        =
    \varnothing,
\]
{\it i.e.}, $\pi' \notin \mathcal{I}(\NC,\Inv(w))$.
\end{proof}

\begin{lemm} \label{lemm:ideals}
    Let $w_1,w_2,\dots$ be a linear extension of the Bruhat order on $W$.  For all $i>1$, the map $w_i W_{\pi} \mapsto \pi$ realizes an isomorphism of poset:
    \[
    \mathcal{I}(\bPF,\{ \{w_1\} , \dots , \{w_{i-1}\}\})
        \cap
    \mathcal{I}(\bPF,\{\{w_{i}\}\})
            \simeq
    \mathcal{I}(\bNC,\Inv(w_i)).
    \]
\end{lemm}

\begin{proof}
    First note that $w_i W_{\pi} \mapsto \pi$ is a bijective poset map from $\mathcal{I}(\PF,\{\{w_{i}\}\})$ to $\NC$, see Lemma~\ref{lemm:pftonc}.  By removing the bottom element, it also gives a bijective poset map from $\mathcal{I}(\bPF,\{\{w_{i}\}\})$ to $\bNC \cup \{e\}$.  To show that a further restriction gives the announced isomorphism, it remains to show that for all $\pi \in \bNC$, we have:
    \begin{equation} \label{eq:equiv}
        w_i W_{\pi} \in \mathcal{I}(\bPF,\{ \{w_1\} , \dots , \{w_{i-1}\} \})
            \Longleftrightarrow
        \pi \in \mathcal{I}(\bNC,\Inv(w_i)).
    \end{equation}
    for any $\pi \in \bNC$.  To do that, consider $w'$ defined as the minimal-length representative in $w_i W_\pi$, so that $w' \leq w_i$ in the Bruhat order and $w' W_\pi = w_i W_\pi$.  Because $w_1,w_2,\dots$ is a linear extension of the Bruhat order, we have $w' = w_j$ for some $j \leq i$. 

    By Lemma~\ref{lemm:equ}, applied on the relation $w_j W_\pi \leq \{w_i\}$, we obtain that the condition on the right-hand side of~\eqref{eq:equiv} is equivalent to $w_j\neq w_i$, {\it i.e.}, $j<i$.

    On the other side, the condition on the left-hand side of~\eqref{eq:equiv} means that there exists $k<i$ such that $w_i W_\pi \leq \{w_k\}$ holds in $\PF$, {\it i.e.}, $w_k \in w_i W_\pi$.  This is equivalent to $w_i$ not being the minimal-length representative in $w_i W_\pi$, {\it i.e.}, $j<i$.  So, both sides of~\eqref{eq:equiv} are indeed equivalent.
\end{proof}

\begin{theo} \label{theo:pftop}
    We have:
    \[
        \Omega(\bPF)
        \simeq
        (\mathbb{S}^{n-1})^{\vee (h-1)^n}.
    \]
    Moreover, we have $\tilde \bH_{n-1}(\bPF) = \sign \otimes \park'$.
\end{theo}

\begin{proof}
    Let $w_1,w_2,\dots$ be a linear extension of the Bruhat order on $W$ as in the previous lemma.  We apply Lemma~\ref{lemm:genshell} with $X = \Omega(\bPF)$ and $A_i = \Omega( \mathcal{I}( \bPF , \{w_i\} ))$.  Note that $A_i$ is topologically trivial because $\{w_i\}$ is a cone point, so the first condition in Lemma~\ref{lemm:genshell} holds.  By Lemma~\ref{lemm:ideals} together Proposition~\ref{prop:idealNC}, the second condition in Lemma~\ref{lemm:genshell} holds as well.  The outcome of Lemma~\ref{lemm:genshell} is that $\Omega(\bPF) \simeq (\mathbb{S}^{n-1})^{\vee \ell}$ with $\ell = \sum_{w\in W} k_w$.  The expression $ \ell = (h-1)^n$ will follow from the second part of the theorem, since $\dim(\park') = (h-1)^n$.  (See Section~\ref{sec:enum} for more about the equality $\sum_{w\in W} k_w = (h-1)^n$.)

    Since that $\Omega(\bPF)$ is a wedge of $(n-1)$-dimensional spheres, it has nonzero homology only in degree $n-1$, so that $\tilde \bchi(\bPF) = (-1)^{n-1} \tilde \bH_{n-1}(\bPF)$.  By Equation~\eqref{eq:relzetachi}, we have the relation $\tilde \bchi(\bPF) = -\bzeta(\PF,-1)$ so that $\tilde \bH_{n-1}(\bPF) = (-1)^n \bzeta(\PF,-1)$.  Using the formula in Proposition~\ref{prop:zetaPF} for $\bzeta(\PF,m)$, the evaluation of $\tilde \bH_{n-1}(\bPF)$ at $w\in W$ is thus:
    \[
        (-1)^n (-h+1)^{n-\ell(w)}
            =
        (-1)^{\ell(w)} (h-1)^{n-\ell(w)},
    \]
    and the two factors in the latter expression are precisely the character evaluations $\sign(w)$ and $\park'(w)$.
\end{proof}

\section{The topology of cluster parking functions}
\label{sec:topcpf}

The results in this section are consequences of Theorem~\ref{theo:pftop}, using the notion of Cohen-Macaulay posets and an associated fiber poset theorem.  We refer to Quillen~\cite{quillen} concerning this material.

\begin{defi}[Quillen~\cite{quillen}]
    Let $P$ be a ranked poset.  Every pair of elements $x,y \in \hat P$ such that $x<y$ defines an open interval $(x,y) := \{ z \in P \;:\; x<z<y\}$ as a subposet of $P$ (or $\hat P$). We say that $P$ is {\it homotopy Cohen-Macaulay} if every nonempty open interval $(x,y)$ satisfies
    \begin{equation}\label{eq:omegaxy}
        \Omega\big((x,y)\big)
        \simeq
        (\mathbb{S}^{\rk(y)-\rk(x)-2})^{\vee k}
    \end{equation}
    for some $k\geq 0$.
\end{defi}

An equivalent definition often seen in the literature is the requirement that the $i$th homotopy group of $\Omega\big((x,y)\big)$ vanishes if $i<\rk(y)-\rk(x)-2$.  This is equivalent to \eqref{eq:omegaxy} (by Proposition~\ref{prop:characwedge}, Hurewicz theorem, and the fact that a $d$-dimensional simplicial complex has no torsion in its $d$th homology group).


Also, there are many variants of Cohen-Macaulay posets, see~\cite{wachs}.


\begin{prop}  \label{prop:pfcohen}
    The poset $\PF$ is homotopy Cohen-Macaulay. 
\end{prop}

\begin{proof}
    Let $x,y\in \hat {\PF}$ such that $x<y$.  To check that Equation~\eqref{eq:omegaxy} holds, we first consider the case $y\in \PF$.  By Lemma~\ref{lemm:pftonc}, the open interval $(x,y)$ is isomorphic to an open interval in $\NC$, so Equation~\eqref{eq:omegaxy} holds because $\NC$ is a shellable poset (see~\cite{athanasiadisbradymccammondwatt}).

    Now, consider the case where $y$ is the extra top element and $x$ the bottom element.  The open interval $(x,y)$ is $\bPF$, so that Equation~\eqref{eq:omegaxy} holds by Theorem~\ref{theo:pftop}.  
    
    In the remaining case, $y$ is the extra top element and $x$ is an arbitrary element of $\PF$.  By Lemma~\ref{lemm:pffilter}, the half-open interval $[x,y) := \{ z \in \PF \;:\; x\leq z <y\}$ is isomorphic to a product $\PF(W_1) \times \PF(W_2) \times \cdots $.  Via an induction hypothesis, we can deduce that $[x,y)$ is homotopy Cohen-Macaulay as a product of homotopy Cohen-Macaulay posets, each of them having a bottom element.  In particular, Equation~\eqref{eq:omegaxy} holds.
\end{proof}

\begin{rema}
    The result about products of posets in the previous paragraph is~\cite[Exercise~5.1.6]{wachs} for the homology version of Cohen-Macaulay posets.  Let us outline the proof for Quillen's definition. Let $P_1$ and $P_2$ be two ranked posets with minima $0_1$ and $0_2$ (respectively). Then, by a result of Quillen~\cite[Proposition~1.9]{quillen}, we have
    \[
        \Omega( (P_1 \times P_2)\backslash{(0_1,0_2)} )
            \simeq
        \Omega( P_1\backslash\{0_1\} ) \star \Omega( P_2\backslash\{0_2\} )
    \]
    where $\star$ is the join operation (see~\cite[Chapter~0]{hatcher}).  Now, if $\Omega( P_i\backslash\{0_i\} ) \simeq (\mathbb{S}^{d_i-1})^{\vee k_i}$ for $i\in \{1,2\}$, we deduce $\Omega( (P_1 \times P_2)\backslash{(0_1,0_2)} ) \simeq (\mathbb{S}^{d_1+d_2-1})^{\vee (k_1k_2)} $.  This follows from the facts that: i) the join is distributive over wedge sums, ii) $\mathbb{S}^{d_1-1} \star \mathbb{S}^{d_2-1} = \mathbb{S}^{d_1+d_2-1}$.  This argument can be generalized to any finite number of factors, and using it on the product $\PF(W_1) \times \PF(W_2) \times \cdots $ shows that Equation~\eqref{eq:omegaxy} holds in the last case of the previous proof.
\end{rema}


\begin{prop}
    The posets $\CPF^{(m)}$ and $\CPF^{(m),+}$ are homotopy Cohen-Macaulay.
\end{prop}

\begin{proof}
Consider the projection $\rho \; : \; \CPF^{(m)} \to \PF$ defined by:
\[
    ( w W_{\underline{f}} , f ) \mapsto w W_{\underline{f}}.
\]
It clearly preserves the order relations and rank functions.  Our goal is to apply Quillen's fiber poset theorem~\cite[Theorem~9.1]{quillen}.  First, we know that $\PF$ is homotopy Cohen-Macaulay by Proposition~\ref{prop:pfcohen}.  It remains to check that each {\it fiber}
\[
    \rho / (w W_{\pi} )
        :=
    \rho^{-1}\big( \mathcal{I}( \PF , \{w W_{\pi}\} ) \big)
\]
is homotopy Cohen-Macaulay, for $w W_\pi \in \PF$.

Explicitly, the fiber is
\[
    \rho / (w W_{\pi} )
        =
    \big\{
        ( w W_{\underline{f}} , f ) \in \CPF^{(m)} \; : \; \underline{f} \leq \pi
    \big\}.
\]
Since the projection $( w W_{\underline{f}} , f ) \mapsto f$ is clearly bijective, we have an isomorphism of posets:
\begin{equation}
    \rho / (w W_{\pi} )
        \simeq
    \big\{ 
        f\in \Delta^{(m)} \;:\; \underline{f} \leq \pi 
    \big\}.
\end{equation}
By Proposition~\ref{prop:subcompl}, the right-hand side (as a simplicial complex) is purely $(\rk(\pi)-1)$-dimensional and shellable.  As shellability implies that a poset is homotopy Cohen-Macaulay, the fibers are so.  We can thus use Quillen's fiber theorem as announced.

The statement about $\CPF^{(m),+}$ is similar, except that the fibers are subcomplexes of $\Delta^{(m),+}$.
\end{proof}

We eventually arrive at our main result:

\begin{theo} \label{theo:topcpf}
    We have:
    \[
        \CPF^{(m)}
            \simeq
        (\mathbb{S}^{n-1})^{\vee (mh+1)^n },
        \qquad
        \CPF^{(m),+}
            \simeq
        (\mathbb{S}^{n-1})^{\vee (mh-1)^n },
    \]
    Moreover, 
    \begin{equation}  \label{eq:bHcpf}
        \tilde \bH_{n-1}(\CPF^{(m)})
            =
        \sign \otimes \park_{m},
        \qquad
        \tilde \bH_{n-1}(\CPF^{(m),+})
            =
        \sign \otimes \park'_{m}.
    \end{equation}
\end{theo}

\begin{proof}
Since $\CPF^{(m)}$ is homotopy Cohen-Macaulay and purely $(n-1)$-dimensional, we have $\CPF^{(m)} \simeq (\mathbb{S}^{n-1})^{\vee k}$ for some $k \geq 0$.  The exact value of $k$ follows from the second part of the corollary, since $\dim(\park_{m}) = (mh+1)^n$.

Since $\CPF^{(m)} \simeq (\mathbb{S}^{n-1})^{\vee k}$, its nontrivial homology is in degree $n-1$ so that $\tilde \bchi(\CPF^{(m)}) = (-1)^{n-1} \tilde \bH_{n-1}(\CPF^{(m)})$.  So, Equation~\eqref{eq:bHcpf} follows from Proposition~\ref{prop:bchicpf}.

The statements about $\CPF^{(m),+}$ are proved similarly.
\end{proof}


\section{\texorpdfstring{$f$-vectors and $h$-vectors}{f-vectors and h-vectors} for cluster parking functions}
\label{ssec:h_vecs}


In this section, we present some numerological results for the face enumeration of many of the simplicial complexes we have considered. Most of the results will follow from similar known theorems on the posets $\PF, \NC$, and the cluster complex, but the particularly nice expressions for the $h$-vectors in Theorem~\ref{thm:h-vector} suggest the existence of nice combinatorial shellings (see also Section~\ref{sec:open}) for the complexes $\CPF^{(m)}$ and $\CPF^{(m),+}$.

Recall that the $f$-vector (\emph{face vector}) of an abstract $(d-1)$-dimensional complex $\Gamma$ is the tuple $f(\Gamma):=(f_{-1},f_0,\ldots,f_{d-1})$ with $f_i$ the number of faces of $\Gamma$ that have dimension $i$ and $f_{-1}=1$. When $\Gamma$ is simplicial, its face vector is often encoded via the related $h$-vector $h(\Gamma)=(h_0,\ldots,h_d)$, a transformation of the $f$-vector defined as follows.  Define the $f$-polynomial and $h$-polynomial by
\[
    f(\Gamma;z):=\sum_{i=-1}^{d-1} f_iz^{i+1},
    \quad \text{ and } \quad
    h(\Gamma;z):=\sum_{i=1}^d h_iz^i.
\]
Then the relation is
\[
    h(\Gamma;z)
        =
    (1-z)^d f\big(\Gamma;\tfrac{z}{1-z}\big).
\]

The $h$-polynomial of a simplicial complex $\Gamma$ always satisfies $h(\Gamma;1)=f_{d-1}$ and can be in certain cases (for instance when $\Gamma$ is shellable) interpreted via a statistic on the facets of $\Gamma$. But even the weaker property that $\Gamma$ is Cohen-Macaulay (which is true in our cases), guarantees that its $h$-vector is non-negative and its leading entry $h_d$ is equal to the dimension of its non-trivial homology (see for example \cite[Ch.~II]{stanley}). In this context, we view Theorem~\ref{thm:h-vector} below as a refinement of the numerological part of Theorem~\ref{theo:topcpf} (see more details in Remark~\ref{Rem:h-vec_conc}).

We need first to fix some terminology from the theory of reflection arrangements (see Orlik and Terao~\cite{orlikterao}). For a general hyperplane arrangement $\mathcal{K}$ in some real space $V$ we will denote by $\mathcal{L}_{\mathcal{K}}$ its \emph{intersection lattice} and we will call its elements $X\in\mathcal{L}_{\mathcal{K}}$ \emph{flats} (they are arbitrary intersections of the hyperplanes of $\mathcal{K}$). The \emph{characteristic polynomial} $\chi(\mathcal{K},t):=\sum_{X\in\mathcal{L}_{\mathcal{K}}}\mu(V,X) t^{\dim(X)}$ records the M{\"o}bius function in $\mathcal{L}_{\mathcal{K}}$ and encodes many combinatorial invariants of $\mathcal{K}$; for instance by Zaslavsky's theorem the quantity $r(\mathcal{K}):=(-1)^{\dim(V)}\cdot \chi(\mathcal{K},-1)$ counts the number of regions determined by the arrangement $\mathcal{K}$.

For any flat $X\in\mathcal{L}_{\mathcal{K}}$ we can define two new hyperplane arrangements, the \emph{localization} $\mathcal{K}_X:=\{H\in\mathcal{K}\ | H\supseteq X\}$ which consists of all hyperplanes that contain $X$ and the \emph{restriction} $\mathcal{K}^X$, which is an arrangement with ambient space $X$ defined as the collection of top dimensional flats properly contained in $X$. More formally, we have $\mathcal{K}^X:=\{H\cap X\ | H\in\mathcal{K},\,H\not\supseteq X\}$.
When $\mathcal{A}$ is the arrangement of hyperplanes of a reflection group $W$, the characteristic polynomials of all its localizations and restrictions factor with non-negative integer roots. When $W$ is irreducible, for any flat $X\in\mathcal{L}_{\mathcal{A}}$ we have that 
\[
\chi(\mathcal{A}^X,t)=\prod_{i=1}^{\dim(X)}(t-b_i^X),
\]
where the $b_i^X$ are positive integers known as the \emph{Orlik-Solomon (OS) exponents} of the flat $X$. The group $W$ acts on the intersection lattice $\mathcal{L}_{\mathcal{A}}$ and the OS exponents depend only on the orbit $[X]\in\mathcal{L}_{\mathcal{A}}/W$. In the case of the symmetric group $\mathfrak{S}_n$ and for any flat $X$, the OS exponents are just the numbers $(b_i^X)=(1,2,\ldots,\dim(X))$.

\begin{prop}\label{Prop:f-vector}
For a finite and irreducible real reflection group $W$ with reflection arrangement $\mathcal{A}$ and intersection lattice $\mathcal{L}_{\mathcal{A}}$, the $f$-polynomials of $\CPF^{(m)}$ and $\CPF^{(m),+}$ are:
\begin{align*}
f\big(\CPF^{(m)};z\big)&=\sum_{X\in\mathcal{L}_{\mathcal{A}}}\prod_{i=1}^{\dim(X)}(mh+1+b_i^X)\cdot z^{\dim(X)}\quad\text{and}\\
f\big(\CPF^{(m),+};z\big)&=\sum_{X\in\mathcal{L}_{\mathcal{A}}}\prod_{i=1}^{\dim(X)}(mh-1+b_i^X)\cdot z^{\dim(X)},
\end{align*}
where $b_i^X$ are the OS exponents for the flat $X$.
\end{prop}
\begin{proof}
As we mentioned earlier, in our previous work \cite{douvropoulosjosuatverges} we gave product formulas for the number of faces in $\Delta^{(m)}$ and $
\Delta^{(m),+}$ that refined the corresponding $f$-vectors. In particular, we showed \cite[Corol.~7.3]{douvropoulosjosuatverges} that the number of faces $f\in \Delta^{(m)}$ (respectively $f\in\Delta^{(m),+}$) such that $W_{\underline{f}}$ is conjugate to a given parabolic subgroup $W_X$ is equal to
\[
\dfrac{1}{[N(X):W_X]}\cdot \prod_{i=1}^{\dim(X)}(mh+1+b_i^X)\quad\text{respectively}\quad\dfrac{1}{[N(X):W_X]}\cdot \prod_{i=1}^{\dim(X)}(mh-1+b_i^X),
\]
where $b_i^X$ are the OS exponents of the flat $X$ and $N(X)$ and $W_X$ its setwise and pointwise stabilizers.

By construction, the fiber of the projection $\CPF^{(m)}\rightarrow \Delta^{(m)}$ over a face $f$ such that $W_{\underline{f}}\sim W_X$ has precisely $[W:W_X]$ elements, all of which have dimension equal to $\dim(X)-1$. Therefore, we will have 
\begin{align*}
f\big(\CPF^{(m)};z\big)&=\sum_{[X]\in\mathcal{L}_{\mathcal{A}}/W}\Big(\dfrac{[W:W_X]}{[N(X):W_X]}\cdot \prod_{i=1}^{\dim(X)}(mh+1+b_i^X)\Big)\cdot z^{\dim(X)}\\
&=\sum_{X\in\mathcal{L}_{\mathcal{A}}}\prod_{i=1}^{\dim(X)}(mh+1+b_i^X)\cdot z^{\dim(X)},
\end{align*}
since there are exactly $[W:N(X)]$-many flats $X$ in the orbit $[X]\in\mathcal{L}_{\mathcal{A}}/W$. The case of $\CPF^{(m),+}$ is analogous and this completes the proof.
\end{proof}

We are now ready to present the main result of this section, computing the $h$-polynomials of $\CPF^{(m)}$ and $\CPF^{(m),+}$. Recall that in a real reflection group $W$ with set of simple generators $S$, some $s \in S$ is a (left) \emph{descent} for an element $w\in W$ if $\ell_S(sw)<\ell_S(w)$ for the Coxeter length $\ell_S$. The descent polynomial of $W$, denoted by $\operatorname{Des}(W;z)$, is the generating polynomial for the descent number over $W$.

\begin{theo}\label{thm:h-vector}
For a finite and irreducible real reflection group $W$ with reflection arrangement $\mathcal{A}$ and intersection lattice $\mathcal{L}_{\mathcal{A}}$, the $h$-polynomials of $\CPF^{(m)}$ and $\CPF^{(m),+}$ are given via
\begin{align*}
h\big(\CPF^{(m)};z\big)&=\sum_{X\in\mathcal{L}_{\mathcal{A}}}\prod_{i=1}^{\dim(X)}(mh+1-b_i^X)\cdot z^{\dim(X)}\cdot \operatorname{Des}(W_X;z)\quad\text{and}\\
h\big(\CPF^{(m),+};z\big)&=\sum_{X\in\mathcal{L}_{\mathcal{A}}}\prod_{i=1}^{\dim(X)}(mh-1-b_i^X)\cdot z^{\dim(X)}\cdot \operatorname{Des}(W_X;z),
\end{align*}
where $b_i^X$ are the OS exponents for the flat $X$ and $\operatorname{Des}(W_X;z)$ is the descent polynomial of $W_X$.  
\end{theo}
\begin{proof}
After Proposition~\ref{Prop:f-vector} and via the $f$- to $h$-vector transformation, we can write the $h$-polynomial of $\CPF^{(m)}$ as
\begin{equation}
h\big(\CPF^{(m)};z\big)=\sum_{X\in\mathcal{L}_{\mathcal{A}}}\prod_{i=1}^{\dim(X)}(mh+1+b_i^X)\cdot z^{\dim(X)}(1-z)^{\operatorname{codim}(X)}.\label{eqn:h-vec_CPF}
\end{equation}
We will expand the factor $\prod(mh+1+b_i^X)$ by using
Kung's identity (as in \cite[eq.~(3.2)]{aguiarbastidasmahajan}) below, which relates the characteristic polynomials of all restrictions and localizations of a hyperplane arrangement $\mathcal{K}$ (here $s$ and $t$ are arbitrary parameters):
\begin{equation}
\chi(\mathcal{K},st)=\sum_{Y\in\mathcal{L}_{\mathcal{K}}}\chi(\mathcal{K}_Y,s)\cdot \chi(\mathcal{K}^Y,t).\label{eqn:Kung}    
\end{equation}
Indeed, applying Kung's identity \eqref{eqn:Kung} on the restricted reflection arrangement $\mathcal{A}^X$ with the values $s=-1$ and $t=mh+1$, and after Zaslavsky's theorem, gives 
\begin{equation}
\prod_{i=1}^{\dim(X)}(mh+1+b_i^X)=\sum_{Y\subseteq X}\prod_{i=1}^{\dim(Y)}(mh+1-b_i^Y)\cdot r(\mathcal{A}^X_Y), \label{eqn:Kung_applic}   
\end{equation}
where recall that $r(\mathcal{A}^X_Y)$ denotes the number of regions of the arrangement $\mathcal{A}^X_Y$. Combining equations \eqref{eqn:Kung_applic} and \eqref{eqn:h-vec_CPF} we get 
\begin{align*}
h\big(\CPF^{(m)};z\big)&=\sum_{X\in\mathcal{L}_{\mathcal{A}}}z^{\dim(X)}(1-z)^{\operatorname{codim}(X)}\cdot \sum_{Y\subseteq X}\prod_{i=1}^{\dim(Y)}(mh+1-b_i^Y)\cdot r(\mathcal{A}^X_Y),\\
&=\sum_{Y\in\mathcal{L}_{\mathcal{A}}}\prod_{i=1}^{\dim(X)}(mh+1-b_i^Y)\cdot \sum_{X\supseteq Y}r(\mathcal{A}_Y^X)\cdot z^{\dim(X)}(1-z)^{\operatorname{codim}(X)}\\
&=\sum_{Y\in\mathcal{L}_{\mathcal{A}}}\prod_{i=1}^{\dim(X)}(mh+1-b_i^Y)\cdot z^{\dim(Y)}\cdot \operatorname{Des}(W_Y;z).
\end{align*}
For the last equation above, notice that the $f$-vector of the Coxeter complex of $W_Y$ is given by $\sum_{X\supseteq Y}r(\mathcal{A}_Y^X)\cdot z^{\dim(X)-\dim(Y)}$ so that the equation is just the $f$- to $h$-transformation for the Coxeter complex of $W_Y$ (we are also using the known facts --see for example \cite[\S1.2]{stembridge}-- that the $h$-polynomial of the complex is symmetric and equals the generating polynomial for the descent statistic over $W_Y$). The argument for $\CPF^{(m),+}$ is the same; the proof is complete.
\end{proof}

\begin{rema}\label{Rem:h-vec_conc}
Notice that Theorem~\ref{thm:h-vector} explicitly shows that the $h$-vectors of $\CPF^{(m)}$ and $\CPF^{(m),+}$ are non-negative --a fact guaranteed by the Cohen-Macaulay property-- since for any flat $X$, we have $h\geq b_i^X$ for all $i$. 

To see that it also generalizes the numerological part of Theorem~\ref{theo:topcpf}, we need only calculate the top entries $h_n$ of the $h$-vectors and show that they agree with the dimension of the non-trivial homology of the complex --again as expected by the Cohen-Macaulay property. For the case of $\CPF^{(m)}$, notice that by Theorem~\ref{thm:h-vector} we have (since the descent polynomials are monic and of degree equal to the rank of the group)
\[
h_n=\sum_{X\in\mathcal{L}_{\mathcal{A}}}\prod_{i=1}^{\dim(X)}(mh+1-b_i^X)=(mh+1)^n,
\]
where the second equality comes from the well known identity of hyperplane arrangements $t^{\dim(V)} = \sum_{X\in\mathcal{L}_{\mathcal{K}}}\chi(\mathcal{K}^X,t)$ applied on the arrangement $\mathcal{A}^X$ and after setting $t=mh+1$. The case of $\CPF^{(m),+}$ is analogous.
\end{rema}


\begin{rema}[Equivariant $h$-vector]
There is a version of Theorem~\ref{thm:h-vector} which generalizes the representation theoretic Equation~\eqref{eq:bHcpf} of Theorem~\ref{theo:topcpf} and which we briefly describe here. 

When there is a group action on the faces of a simplicial complex $\Gamma$, its $f$-vector can naturally be replaced by a tuple of permutation representations $\mathbf{f}(\Gamma):=(\mathbf{f}_{-1},\mathbf{f}_0,\ldots,\mathbf{f}_{d-1})$ and then one can define an equivariant $h$-vector as a virtual character following the $f$- to $h-$ transformation recipe (as in \cite{stanley_posets} or see also the exposition in \cite[\S2.4]{adamsreiner}). The generalization of Theorem~\ref{thm:h-vector} is given in terms of an equivariant version of the descent polynomial of a real reflection group $(W,S)$ coming from Solomon's descent algebra and defined as
\[
\mathbf{Des}(W;z):=\sum_{I\subseteq S}\bm{\beta_I} z^{|I|},\quad\text{where}\quad\bm{\beta_I}:=\sum_{J\supseteq S\setminus I}(-1)^{|S|-|I|-|J|}\big\uparrow^W_{\langle J\rangle}\mathrm{triv}.
\]
Here the $\bm{\beta_I}$ are the homology representations of $W$ on the top homology of the $I$-selected subcomplex of the Coxeter complex (as in \cite[\S6]{bjorner}) and are also known \cite[\S6]{miller} as ribbon representations because in the case of the symmetric group $S_n$ they correspond to skew Schur characters indexed by ribbon shapes.

Then, with essentially the same proof as in Theorem~\ref{thm:h-vector} one can show that
\begin{equation}
\mathbf{h}\big(\CPF^{(m)};z\big)=\sum_{[X]\in\mathcal{L}_{\mathcal{A}}/W}\Big(\dfrac{\prod_{i=1}^{\dim(X)}(mh+1-b_i^X)}{[N(X):W_X]} \Big)\cdot z^{\dim(X)}\cdot \big\uparrow^W_{W_X}\mathbf{Des}(W_X;z),\label{eq:equiv_h_vec}
\end{equation}
and similarly for $\CPF^{(m),+}$. Notice that in $\mathbf{Des}(W_X;z)$ the highest degree coefficient is always the sign representation, so that equation \eqref{eq:bHcpf} can be deduced from \eqref{eq:equiv_h_vec} after the usual decomposition of the parking space into parabolic cosets.
\end{rema}

In the case of the symmetric group $\mathfrak{S}_n$, the $h$-polynomial of $\CPF_n$ has appeared before in a different context.  Archer et al. \cite{archer} introduced quasi-Stirling permutations as a generalization of the Stirling permutations of Gessel and Stanley. The \emph{quasi-Stirling polynomial} $\overline{Q}_n(z)$ is the generating polynomial for the descent number over quasi-Stirling permutations of length $2n$. Elizalde \cite[Thm.~2.5]{elizalde} gave the following formula for $\overline{Q}_n(z)$:
\begin{equation}
\overline{Q}_n(z)=(1-z)^{2n+1} \sum_{i\geq 0} \frac{i^n}{n+1} \binom{n+i}{i} z^i.\label{eqn:defn:q-stirl}
\end{equation}

\begin{prop}
The $h$-vector of $\CPF_n$ and the quasi-Stirling polynomial $\overline{Q}_n(z)$ satisfy 
\[
\overline{Q}_n(z)=z\cdot h\big(\CPF_n;z\big).
\]
\end{prop}
\begin{proof}
In the case of the symmetric groups $W=\mathfrak{S}_n$ and $m=1$, Equation~\eqref{eqn:h-vec_CPF} from the proof of Theorem~\ref{thm:h-vector} can be rewritten as 
\[
\dfrac{z\cdot h(\CPF_n;z)}{(1-z)^{2n+1}}=\sum_{k=1}^n S(n,k)\cdot (n+2)\cdots (n+k)\cdot z^k(1-z)^{-1-n-k},
\]
where $S(n,k)$ are the Stirling numbers of the second kind (since flats of the reflection arrangement of $\mathfrak{S}_n$ correspond to set partitions and since $b_i^X=i$ in $S_n$). To compare with equation \eqref{eqn:defn:q-stirl} we compute the coefficient of $z^i$ in the expression above; it equals
\[
\sum_{j=1}^i S(n,j)\cdot (n+2)\cdots (n+j)\cdot \binom{n+i}{i-j}=\sum_{j=1}^iS(n,j)\cdot\dfrac{i!}{(i-j)!}\cdot\dfrac{1}{n+1}\binom{n+i}{i},
\]
which completes the proof after the standard identity $\sum_{j=1}^iS(n,j)\cdot i!/(i-j)!=i^n$.
\end{proof}

On a different direction we may compare $\CPF$ and $\PF$.
Let us begin with a result from~\cite{delcroixogerjosuatvergesrandazzo} (in the case of the symmetric group).  The {\it Whitney numbers of the first kind} 
of $\PF$ are defined by
\[
    w_i(\PF)
        =
    \sum_{wW_\pi \in \PF,\; \rk(wW_\pi) = i} \mu_{\PF}(wW_\pi)
\]
for $0\leq i \leq n$.  Here, $\mu_{\PF}$ is the Möbius function of $\PF$. 

\begin{prop}
    We have $f_i(\CPF^+) = (-1)^i w_i(\PF)$. 
\end{prop}

\begin{proof}
    This is exactly the same proof as~\cite[Proposition~5.7]{delcroixogerjosuatvergesrandazzo}.  By Lemma~\ref{lemm:pftonc}, the interval $[W,wW_\pi]$ in $\PF$ is isomorphic to the interval $[c,\pi]$ in $\NC$, so that $\mu_{\PF}(wW_\pi) = \mu_{\NC}(\pi)$.  This number $\mu_{\NC}(\pi)$ is also the number of $f\in \Delta^+$ such that $\underline{f} = \pi$. (This follows from two well-known interpretations of $\Cat^+$: either the Möbius number of $\NC$ or the number of facets of $\Delta^+$.)  We thus have
    \begin{align*}
        (-1)^iw_i(\PF)
            &=
        \sum_{wW_\pi \in \PF,\; \rk(wW_\pi) = i} (-1)^i\mu_{\NC}(\pi) \\
            &=
        \sum_{wW_\pi \in \PF,\; \rk(wW_\pi) = i} \#\Big\{ f \in \Delta^+ \;:\; 
        \underline{f} = \pi \Big\}
        = f_i(\CPF^+).
    \end{align*}
\end{proof}

\section{New combinatorial objects equinumerous with parking functions}

\label{sec:enum}

Recall that we defined, for each $w\in W$:
    \begin{align} \label{def:kw2}
        k_w :=
        \# \big\{
            f \in \Delta^+_{n-1} \;:\; 
            f \subset \Inv(w)
        \big\}.
    \end{align}
By a double counting of the number of spheres in the homotopy type of the parking function poset $\PF$, we got the identity 
\begin{equation} \label{eq:idkw}
    \sum_{w\in W} k_w =(h-1)^n
\end{equation}
(see the proof of Theorem~\ref{theo:pftop}).  We investigate here some enumerative consequences of this identity.  Through this section, we use the inversion set defined as a set of positive roots:  for $w\in W$,
\[
    \Inv(w) 
        :=
    w^{-1}(\Phi_-) \cap \Phi_+.
\]

\begin{defi}
    Let $\LC$ denote the set of {\it $W$-labeled clusters}, defined by:
    \[
        \LC 
            :=
        \big\{
            (f,w) \in \Delta_{n-1} \times W
                \;:\;
            f \cap \Inv(w) = \varnothing
        \big\}.
    \]
    Moreover, let $\LC^+ := \LC \mathrel{\cap} (\Delta_{n-1}^+ \times W)$ denote the set of {\it $W$-labeled positive clusters}.
\end{defi}

The term ``labeled cluster'' occasionally appears in the cluster algebra literature, as the cluster in a labeled seed.  It is thus better that the terminology refers to the group $W$.

Let us describe explicitly $\mathfrak{S}_n$-labeled clusters (type $A_{n-1}$) in the case of the linear Coxeter element $c=(2,3,\dots,n-1,n,1)$.  It is convenient to represent a type $A_{n-1}$ cluster $f$ as follows: draw $n$ dots on an horizontal line, draw an arch above the line from the $i$th dot to the $j$th dot if $f$ contains the positive root $e_i - e_j$ where $i<j$.  This representation can be characterized by the following properties:
\begin{itemize}
    \item the arches are noncrossing and alternating, {\it i.e.}, there is no two arches $(i,j)$ and $(k,l)$ such that $i<k \leq j < l$,
    \item each connected component defined by the arches is an interval of consecutive dots.
\end{itemize}
Such a system of arches is, for example, 
\[
\begin{tikzpicture}[scale=0.5]
   \tikzstyle{ver} = [circle, draw, fill, inner sep=0.5mm]
   \tikzstyle{edg} = [line width=0.6mm]
   \node[ver] at (1,0) {};
   \node[ver] at (2,0) {};
   \node[ver] at (3,0) {};
   \node[ver] at (4,0) {};
   \node[ver] at (5,0) {};
   \node[ver] at (6,0) {};
   \node[ver] at (7,0) {};
   \node[ver] at (8,0) {};
   \node[ver] at (9,0) {};
   \node[ver] at (10,0) {};
   \node[ver] at (11,0) {};
   \node[ver] at (12,0) {};
   \node[ver] at (13,0) {};
   \node[ver] at (14,0) {};
   \node[ver] at (15,0) {};
   \node[ver] at (16,0) {};
   \node[ver] at (17,0) {};
   \draw[edg] (4,0) arc [start angle=0, end angle=180, radius=1.5];
   \draw[edg] (4,0) arc [start angle=0, end angle=180, radius=1];
   \draw[edg] (4,0) arc [start angle=0, end angle=180, radius=0.5];
   \draw[edg] (6,0) arc [start angle=0, end angle=180, radius=2.5];
   \draw[edg] (6,0) arc [start angle=0, end angle=180, radius=0.5];
   \draw[edg] (8,0) arc [start angle=0, end angle=180, radius=0.5];
   \draw[edg] (10,0) arc [start angle=0, end angle=180, radius=0.5];
   \draw[edg] (10,0) arc [start angle=0, end angle=180, radius=1.5];
   \draw[edg] (12,0) arc [start angle=0, end angle=180, radius=0.5];
   \draw[edg] (13,0) arc [start angle=0, end angle=180, radius=1];
   \draw[edg] (15,0) arc [start angle=0, end angle=180, radius=0.5];
   \draw[edg] (16,0) arc [start angle=0, end angle=180, radius=1];
   \draw[edg] (17,0) arc [start angle=0, end angle=180, radius=1.5];
   \draw[edg] (17,0) arc [start angle=0, end angle=180, radius=3];
 \end{tikzpicture}\,.
\]
Now, a $\mathfrak{S}_n$-labeled cluster is obtained by labelling the $n$ dots with integers in $\{1,\dots,n\}$, in such a way that for each arch, the label at its left endpoint is smaller than the label at its right endpoint.  For example, the 16 $\mathfrak{S}_3$-labeled clusters are:
\begin{align*}
   \begin{tikzpicture}[scale=0.5]
   \tikzstyle{ver} = [circle, draw, fill, inner sep=0.5mm]
   \tikzstyle{edg} = [line width=0.6mm]
   \node[ver] at (1,0) {};
   \node[ver] at (2,0) {};
   \node[ver] at (3,0) {};
   \node at (1,-0.6){1};
   \node at (2,-0.6){2};
   \node at (3,-0.6){3};
   \draw[edg] (3,0) arc [start angle=0, end angle=180, radius=1];
   \draw[edg] (2,0) arc [start angle=0, end angle=180, radius=0.5];
   \end{tikzpicture}, \quad
   \begin{tikzpicture}[scale=0.5]
   \tikzstyle{ver} = [circle, draw, fill, inner sep=0.5mm]
   \tikzstyle{edg} = [line width=0.6mm]
   \node[ver] at (1,0) {};
   \node[ver] at (2,0) {};
   \node[ver] at (3,0) {};
   \node at (1,-0.6){1};
   \node at (2,-0.6){3};
   \node at (3,-0.6){2};
   \draw[edg] (3,0) arc [start angle=0, end angle=180, radius=1];
   \draw[edg] (2,0) arc [start angle=0, end angle=180, radius=0.5];
   \end{tikzpicture},\;
   \begin{tikzpicture}[scale=0.5]
   \tikzstyle{ver} = [circle, draw, fill, inner sep=0.5mm]
   \tikzstyle{edg} = [line width=0.6mm]
   \node[ver] at (1,0) {};
   \node[ver] at (2,0) {};
   \node[ver] at (3,0) {};
   \node at (1,-0.6){1};
   \node at (2,-0.6){2};
   \node at (3,-0.6){3};
   \draw[edg] (3,0) arc [start angle=0, end angle=180, radius=1];
   \draw[edg] (3,0) arc [start angle=0, end angle=180, radius=0.5];
   \end{tikzpicture}, \quad
   \begin{tikzpicture}[scale=0.5]
   \tikzstyle{ver} = [circle, draw, fill, inner sep=0.5mm]
   \tikzstyle{edg} = [line width=0.6mm]
   \node[ver] at (1,0) {};
   \node[ver] at (2,0) {};
   \node[ver] at (3,0) {};
   \node at (1,-0.6){2};
   \node at (2,-0.6){1};
   \node at (3,-0.6){3};
   \draw[edg] (3,0) arc [start angle=0, end angle=180, radius=1];
   \draw[edg] (3,0) arc [start angle=0, end angle=180, radius=0.5];
   \end{tikzpicture}, \quad
   \begin{tikzpicture}[scale=0.5]
   \tikzstyle{ver} = [circle, draw, fill, inner sep=0.5mm]
   \tikzstyle{edg} = [line width=0.6mm]
   \node[ver] at (1,0) {};
   \node[ver] at (2,0) {};
   \node[ver] at (3,0) {};
   \node at (1,-0.6){1};
   \node at (2,-0.6){2};
   \node at (3,-0.6){3};
   \draw[edg] (2,0) arc [start angle=0, end angle=180, radius=0.5];
   \end{tikzpicture}, \quad
   \begin{tikzpicture}[scale=0.5]
   \tikzstyle{ver} = [circle, draw, fill, inner sep=0.5mm]
   \tikzstyle{edg} = [line width=0.6mm]
   \node[ver] at (1,0) {};
   \node[ver] at (2,0) {};
   \node[ver] at (3,0) {};
   \node at (1,-0.6){1};
   \node at (2,-0.6){3};
   \node at (3,-0.6){2};
   \draw[edg] (2,0) arc [start angle=0, end angle=180, radius=0.5];
   \end{tikzpicture}, \quad
   \begin{tikzpicture}[scale=0.5]
   \tikzstyle{ver} = [circle, draw, fill, inner sep=0.5mm]
   \tikzstyle{edg} = [line width=0.6mm]
   \node[ver] at (1,0) {};
   \node[ver] at (2,0) {};
   \node[ver] at (3,0) {};
   \node at (1,-0.6){2};
   \node at (2,-0.6){3};
   \node at (3,-0.6){1};
   \draw[edg] (2,0) arc [start angle=0, end angle=180, radius=0.5];
   \end{tikzpicture}, \quad
   \begin{tikzpicture}[scale=0.5]
   \tikzstyle{ver} = [circle, draw, fill, inner sep=0.5mm]
   \tikzstyle{edg} = [line width=0.6mm]
   \node[ver] at (1,0) {};
   \node[ver] at (2,0) {};
   \node[ver] at (3,0) {};
   \node at (1,-0.6){1};
   \node at (2,-0.6){2};
   \node at (3,-0.6){3};
   \draw[edg] (3,0) arc [start angle=0, end angle=180, radius=0.5];
   \end{tikzpicture}, \\
   \begin{tikzpicture}[scale=0.5]
   \tikzstyle{ver} = [circle, draw, fill, inner sep=0.5mm]
   \tikzstyle{edg} = [line width=0.6mm]
   \node[ver] at (1,0) {};
   \node[ver] at (2,0) {};
   \node[ver] at (3,0) {};
   \node at (1,-0.6){2};
   \node at (2,-0.6){1};
   \node at (3,-0.6){3};
   \draw[edg] (3,0) arc [start angle=0, end angle=180, radius=0.5];
   \end{tikzpicture}, \quad
   \begin{tikzpicture}[scale=0.5]
   \tikzstyle{ver} = [circle, draw, fill, inner sep=0.5mm]
   \tikzstyle{edg} = [line width=0.6mm]
   \node[ver] at (1,0) {};
   \node[ver] at (2,0) {};
   \node[ver] at (3,0) {};
   \node at (1,-0.6){3};
   \node at (2,-0.6){1};
   \node at (3,-0.6){2};
   \draw[edg] (3,0) arc [start angle=0, end angle=180, radius=0.5];
   \end{tikzpicture}, \quad
   \begin{tikzpicture}[scale=0.5]
   \tikzstyle{ver} = [circle, draw, fill, inner sep=0.5mm]
   \tikzstyle{edg} = [line width=0.6mm]
   \node[ver] at (1,0) {};
   \node[ver] at (2,0) {};
   \node[ver] at (3,0) {};
   \node at (1,-0.6){1};
   \node at (2,-0.6){2};
   \node at (3,-0.6){3};
   \end{tikzpicture}, \quad
   \begin{tikzpicture}[scale=0.5]
   \tikzstyle{ver} = [circle, draw, fill, inner sep=0.5mm]
   \tikzstyle{edg} = [line width=0.6mm]
   \node[ver] at (1,0) {};
   \node[ver] at (2,0) {};
   \node[ver] at (3,0) {};
   \node at (1,-0.6){2};
   \node at (2,-0.6){1};
   \node at (3,-0.6){3};
   \end{tikzpicture}, \quad
   \begin{tikzpicture}[scale=0.5]
   \tikzstyle{ver} = [circle, draw, fill, inner sep=0.5mm]
   \tikzstyle{edg} = [line width=0.6mm]
   \node[ver] at (1,0) {};
   \node[ver] at (2,0) {};
   \node[ver] at (3,0) {};
   \node at (1,-0.6){3};
   \node at (2,-0.6){1};
   \node at (3,-0.6){2};
   \end{tikzpicture}, \quad
   \begin{tikzpicture}[scale=0.5]
   \tikzstyle{ver} = [circle, draw, fill, inner sep=0.5mm]
   \tikzstyle{edg} = [line width=0.6mm]
   \node[ver] at (1,0) {};
   \node[ver] at (2,0) {};
   \node[ver] at (3,0) {};
   \node at (1,-0.6){1};
   \node at (2,-0.6){3};
   \node at (3,-0.6){2};
   \end{tikzpicture}, \quad
   \begin{tikzpicture}[scale=0.5]
   \tikzstyle{ver} = [circle, draw, fill, inner sep=0.5mm]
   \tikzstyle{edg} = [line width=0.6mm]
   \node[ver] at (1,0) {};
   \node[ver] at (2,0) {};
   \node[ver] at (3,0) {};
   \node at (1,-0.6){2};
   \node at (2,-0.6){3};
   \node at (3,-0.6){1};
   \end{tikzpicture}, \quad
   \begin{tikzpicture}[scale=0.5]
   \tikzstyle{ver} = [circle, draw, fill, inner sep=0.5mm]
   \tikzstyle{edg} = [line width=0.6mm]
   \node[ver] at (1,0) {};
   \node[ver] at (2,0) {};
   \node[ver] at (3,0) {};
   \node at (1,-0.6){3};
   \node at (2,-0.6){2};
   \node at (3,-0.6){1};
   \end{tikzpicture}.
\end{align*}

\begin{prop} \label{prop:fccard}
    We have $\# \LC = (h+1)^n$, and $\# \LC^+ = (h-1)^n$.
\end{prop}

\begin{proof}
    We begin with positive $W$-labeled clusters.  Let $w\in W$, and note that we have a set partition $\Phi_+ = \Inv(w) \mathrel{\uplus} \Inv(w w_\circ)$ (where $w_\circ$ is the longest element of $W$).  Consequently, for $f\in \Delta^+$ there is an equivalence: $f \mathrel{\cap} \Inv(w) =\varnothing$ iff $f \subset \Inv(ww_\circ)$. The number of $f\in \Delta_{n-1}^+$ such that $(f,w)$ is a positive $W$-labeled cluster is thus $k_{ww_\circ}$ (by the definition in~\eqref{def:kw2}).  By summing over $w$, the number of positive labeled clusters is:
    \[
        \# \LC^+
            =
        \sum_{w\in W} k_{ww_\circ}
            =
        \sum_{w\in W} k_{w}
            =
        (h-1)^n,
    \]
    by Equation~\eqref{eq:idkw}.

    To get the other result, first recall that there is a bijection $W \to (W/W_I) \times W_I$ given by 
    \[
        w \mapsto ( wW_I, w_I )
    \]
    where $w_I \in W_I$ is defined by $\Inv(w_I) = \Inv(w) \mathrel{\cap} \Phi^I_+$ (and $\Phi^I_+ \subset \Phi_+$ is the set of positive roots of $W_I$).  Using that, we can give a bijection:
    \begin{equation} \label{bij:fc}
        \LC \to \biguplus_{I\subset S}
        W/W_I \times \LC^+(W_I).
    \end{equation}
    Explicitly, let $(f,w) \in \LC$ and let $I \subset S$ be such that the negative roots in $f$ are indexed by $S\backslash I$.  Then, this $(f,w)$ is associated to $(wW_I, (f\cap\Phi_+, w_I) ) \in W/W_I \times \LC^+(W_I)$.  It is straightforward to check that this defines a bijection as in~\eqref{bij:fc}.
    
    Now, there is also a bijection:
    \begin{equation} \label{bij:pf}
        \PF \to \biguplus_{I\subset S}
        W/W_I \times \PF'(W_I).
    \end{equation}
    It is a consequence of an identity on characters given in~\cite[Proposition~4.1]{douvropoulosjosuatverges}.  It can be given explicitly by sending $w W_\pi$ to $(wW_I , w_I W_\pi)$ where $I\subset S$ is minimal such that $\pi\in W_I$.
    
    Using the two bijections in~\eqref{bij:fc} and~\eqref{bij:pf}, together with $ \# \LC^+(W_I) = \#  \PF'(W_I)$, we get $\# \LC = \# \PF$.  It follows that $\LC = (h+1)^n$.
\end{proof}

Let us mention that the naive Fuß generalization of the previous result doesn't hold (more precisely: counting pairs $(f,w) \in \Delta^{(m)}_{n-1} \times W$ such that $f \cap \Inv(w) = \varnothing$ does not give $(mh+1)^n$, or any interesting number whatsoever).  

Let us end this section with a small remark.  Consider the quantities:
\begin{align}
    \ell_w &:= 
    \# \big\{
        f \in \Delta_{n-1} \;:\; f \mathrel{\cap} \Inv(w) = \varnothing
    \big\}, \\
    m_w &:= 
    \# \big\{
        \pi \in \NC \;:\; W_\pi \mathrel{\cap} \Inv(w) = \varnothing
    \big\}.
\end{align}
Note that $\ell_w = k_{w w_\circ}$.  On the other side, $m_w$ is the number of $w W_\pi \in \PF$ where $w$ is a minimal length coset representative.  There seems to be some similarities between these numbers.  We have $\ell_e=m_e$ (this is the Catalan number of $W$), $\ell_{w_\circ} = m_{w_\circ} = 0$, and $\sum_{w\in W} \ell_w = \sum_{w\in W} m_w = (h+1)^n$.  Very often, we have $\ell_w = m_w$.  However, this equality does not hold in general.  When $W=\mathfrak{S}_4$ and $c=(1234)$, with $w=2413$ we have $\ell_w = 4$ and $m_w=5$.

In the first version of this work, an open problem was to find a direct bijection between $\LC$ and parking functions.  It turns out that there is a simple answer.  To $(f,w) \in \LC$, we associate $(\varphi(f),w)$ where $\varphi$ is the bijection between clusters and noncrossing partitions given in~\cite{athanasiadisbradymccammondwatt}.  By definition of this bijection, we have $\varphi(f) \leq \prod f$.  By convexity of inversion sets, $\Inv(w)$ is disjoint from $W_{\varphi(f)}$, so $w$ is a minimal length representative in $W/W_{\varphi(f)}$.  The conclusion is that the bijection is given by $(f,w) \mapsto wW_{\varphi(f)}$.

\section{\texorpdfstring{Interpretation via $m$-Catalan arrangements}{Interpretation via m-Catalan arrangements}}

There is a well-known encoding of the $h$-vectors for the complexes $\Delta^{(m)}$ and $\Delta^{(m),+}$ as statistics on the dominant regions of the $m$-Catalan arrangement. We will present in this section a generalization for our $\CPF^{(m)}$ and $\CPF^{(m),+}$ complexes.

Recall that for a crystallographic group $W$ with root system $\Phi$ and a choice of positive roots $\Phi^{+}$, its reflection arrangement $\mathcal{A}_W$ and its $m$-Catalan arrangement $\operatorname{Cat}^{(m)}(W)$ are defined via the equations of their hyperplanes as
\[
\operatorname{Cat}^{(m)}(W):=\big\{\langle\alpha,x\rangle=k\ |\ \alpha\in\Phi^{+}, k\in \mathbb{Z},\ -m\leq k \leq m\big\}\quad\text{and}\quad \mathcal{A}_W=\operatorname{Cat}^{(0)}(W).
\]
The hyperplanes of the reflection arrangement $\mathcal{A}_W$ divide the ambient space into $|W|$-many connected components; there is a unique one for which $\langle \alpha,x\rangle\geq 0$ for all $\alpha\in\Phi^{+}$ which we call the fundamental chamber of $W$. Now, since the reflection arrangement of $W$ is a subarrangement of the $m$-Catalan arrangement, all regions of $\operatorname{Cat}^{(m)}(W)$ must lie in some chamber of $\mathcal{A}_W$. We will call those regions that lie in the fundamental chamber of $\mathcal{A}_W$ \emph{dominant regions} and denote them by $\mathcal{DRC}^{(m)}(W)$, while we will write  $\mathcal{RC}^{(m)}(W)$ for the set of \emph{all} regions of $\operatorname{Cat}^{(m)}(W)$. 

The hyperplanes of $\operatorname{Cat}^{(m)}(W)$ that support facets of a region $R\in \mathcal{RC}^{(m)}(W)$ are called the \emph{walls} of $R$. Those walls of $R$ that separate it from the origin, while they themselves \emph{do not contain the origin}, will be called the \emph{floors} of $R$ (see Fig.~\ref{fig:m-floors}). A floor of some region $R$ in $\operatorname{Cat}^{(m)}(W)$ that corresponds to a hyperplane of the form $\langle \alpha,x\rangle=m$ for some $\alpha\in\Phi^{+}$ will be called an $m$-floor. The number of $m$-floors of a region $R$ is denoted by $\operatorname{mfl}(R)$.

The statistic $\operatorname{mfl}(R)$ is particularly interesting because its distribution on the set of dominant regions $\mathcal{DRC}^{(m)}(W)$ encodes the $h$-polynomial of the generalized cluster complex $\Delta^{(m)}$ of $W$. The following statement is a combination of \cite[Thm.~1]{thiel_eur} and \cite[Thm.~1.2]{athan_trans}. 

\begin{prop}\label{prop:h_via_mfl}
For any Weyl group $W$ of rank $n$ and any $m\in\mathbb{Z}_{\geq 0}$, the $h$-polynomial of the complex $\Delta^{(m)}$ is given via
\begin{equation}
h(\Delta^{(m)};z)=\sum_{R\in \mathcal{DRC}^{(m)}(W)}z^{n-\operatorname{mfl}(R)}
\end{equation}
where $\operatorname{mfl}(R)$ is the number of $m$-floors of the region $R$.
\end{prop}

\begin{figure}
\centering
\includegraphics[scale=0.5]{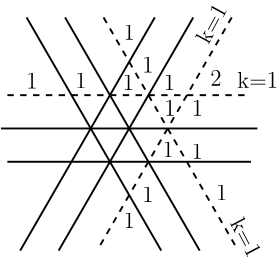}\hspace{3cm}
\includegraphics[scale=0.35]{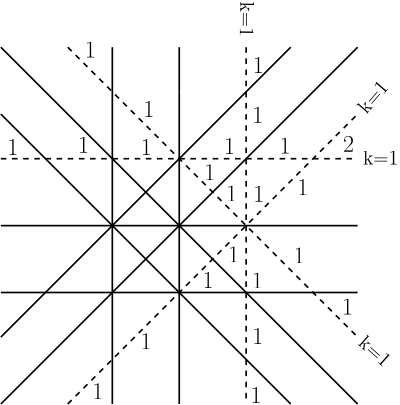}
\caption{\small Regions with $\operatorname{mfl}(R)>0$ in $\operatorname{Cat}^{(1)}(A_2)$ and $\operatorname{Cat}^{(1)}(B_2)$}\label{fig:m-floors}
\end{figure}

Now, the collection of $m$-floors of a region $R$ allows us to define a finer statistic for it. Every $m$-floor is parallel to a unique, central, reflection hyperplane of $W$, and the intersection of the $m$-floors of $R$ is parallel to a unique flat $X$ of the reflection arrangement $\mathcal{A}$ of $W$. We call the $W$-orbit $[X]\in \mathcal{L}_{\mathcal{A}}/W$ of this flat, the \emph{parabolic type} of the region $R$. The following lemma describes the distribution of the $m$-floors statistic in the $W$-orbit of a dominant region $R$ in terms of its parabolic type.

\begin{lemm}
For a Weyl group $W$ and any dominant region $R\in\mathcal{RC}^{(m)}(W)$ of parabolic type $[X]$, we have that ($\operatorname{Des}(W_X;z)$ denotes the descent polynomial of $W_X$ as in Section~\ref{ssec:h_vecs})
\[
\sum_{w\in W}z^{\operatorname{mfl}(w\cdot R)}=[W:W_X]\cdot \operatorname{Des}(W_X;z).
\]
\end{lemm}
\begin{proof}
Let $D(R)$ denote the set of positive roots $\alpha\in\Phi^+$ such that the hyperplane $\langle \alpha,x\rangle=m$ is a wall of $R$. Because $R$ is dominant all its other walls correspond to hyperplanes of the form $\langle\alpha,x\rangle=k$ for some $0\leq k<m$. Now, if $\langle\alpha,x\rangle=k$ is a wall of $R$, then $\langle w\cdot \alpha,x\rangle=k$ is the corresponding wall of $w\cdot R$. Thus, the number of $m$-floors of the region $wR$ is given via
\[
\operatorname{mfl}(w\cdot R)=\operatorname{mfl}(R)-\#\{\alpha\in D(R):\,w\cdot\alpha\in\Phi^{-}\}.
\]
We can assume by \cite{sommers} that after some conjugation the set $D(R)$ is a subset $J\subset S$ of the set $S$ of simple roots of $W$, and $W_X=W_J$. Now, there is a well known unique factorization result $W=W^J\cdot W_J$, where $W^J$ denotes the set of minimum length coset representatives for $W/W_J$. For each $J\subset S$, every element $w\in W$ can be (uniquely) expressed as $w=w^J\cdot w_J$, with $w^J\in W^J$ and $w_J\in W_J$, and where the descents of $w^J$ lie in $S\setminus J$. This means that all elements $w^J$ will send the set of positive simple roots $J$ to positive roots, while the elements $w_J$ will send precisely  their descent sets $\operatorname{des}_S(w_J)$ to negative roots. That is, we will have 
\[
\operatorname{mfl}(w^Jw_J\cdot R)=|J|-\#\operatorname{des}_S(w_J).
\]
We can now rewrite the statistic for the distribution of $m$-floors as:
\[
\sum_{w=w^Jw_J\in W}z^{\operatorname{mfl}(w\cdot R)}=[W:W_J]\cdot\sum_{g\in W_J}z^{|J|-\#\operatorname{des}_S(w_J)}=[W:W_J]\cdot \operatorname{Des}(W_J;z),
\]
where the last equality is the symmetry of descent polynomials. This completes the proof.
\end{proof}

We are now ready to prove the main theorem of this section.

\begin{theo}\label{thm:CPF_Cat}
For every Weyl group $W$ of rank $n$ and every $m\in\mathbb{Z}_{\geq 0}$, the $h$-polynomial of the complex $\CPF^{(m)}$ is given via 
\begin{equation}
h(\CPF^{(m)};z)=\sum_{R\in \mathcal{RC}^{(m)}(W)}z^{n-\operatorname{mfl}(R)} \label{eq:thm_CPF_Cat}
\end{equation}
where $\operatorname{mfl}(R)$ is the number of $m$-floors of the region $R$.
\end{theo}
\begin{proof}
The proof is a direct application of the previous lemma. We start by emphasizing that the statement of Proposition~\ref{prop:h_via_mfl} is compatible with the notion of parabolic type. Indeed \cite{athan_trans} implies that in fact
\[
\sum_{R\in \mathcal{DRC}^{(m)}(W)}z^{n-\operatorname{mfl}(R)}=\sum_{[X]\in\mathcal{L}_{\mathcal{A}}/W}\Big(\dfrac{\prod_{i=1}^{\dim(X)}(mh+1-b_i^X)}{[N(X):W_X]} \Big)\cdot z^{\dim(X)}.
\]
Now, combinig this with the previous lemma, we get that 
\begin{align*}
\sum_{R\in \mathcal{RC}^{(m)}}z^{n-\operatorname{mfl}(R)}&=\sum_{[X]\in\mathcal{L}_{\mathcal{A}}/W}\Big(\dfrac{\prod_{i=1}^{\dim(X)}(mh+1-b_i^X)}{[N(X):W_X]} \Big)\cdot z^{\dim(X)}\cdot [W:W_X]\cdot\operatorname{Des(W_X;z)}\\
&=\sum_{X\in\mathcal{L}_{\mathcal{A}}}\prod_{i=1}^{\dim(X)}(mh+1-b_i^X)\cdot z^{\dim(X)}\cdot \operatorname{Des}(W_X;z),
\end{align*}
where the second equality is using that there are $[W:N(X)]$ many flats in the $W$-orbit $[X]$. This last expression is equal to the $h$-polynomial of the complex $\CPF^{(m)}$ after Theorem~\ref{thm:h-vector} and the proof is complete.
\end{proof}

\begin{rema}
Theorem~\ref{thm:CPF_Cat} has a counterpart for the complex $\CPF^{(m),+}$ where the summation is over the \emph{bounded} regions of $\operatorname{Cat}^{(m)}(W)$ and $m$-floors are replaced with $m$-ceilings as in \cite{thiel}. The proof is exactly the same.

We remark also that our current understanding of Theorem~\ref{thm:CPF_Cat} is purely numerological; the two sides of equation \eqref{eq:thm_CPF_Cat} have been computed separately and shown to agree. It would be nice to have a conceptual explanation for the connection between them.
\end{rema}

\section{Open problems}
\label{sec:open}





\begin{open}
    Complete Section~\ref{sec:combmod} by giving a combinatorial model for type $D$ cluster parking functions (possibly building on the combinatorial model of type $D$ clusters in~\cite[Section~5.3]{fominreading}).  It could be useful to have combinatorial models for the three infinite families of the finite type classification.
\end{open}


\begin{open}
    Complete Proposition~\ref{prop:flag} by showing that $\CPF^{(m)}$ and $\CPF^{(m),+}$ are flag complexes.  This could be done by proving Conjecture~\ref{conj:helly} about the Helly property (which might be interesting to investigate on its own).
\end{open}

\begin{open}
    Let $\mathbb{Z}_{mh+2}$ denote the cyclic group or order $mh+2$.  Via the rotation $\mathcal{R}$, $\mathbb{Z}_{mh+2}$ acts on $\Delta^{(m)}$ by automorphisms (see~\cite[Theorem~3.4]{fominreading}).  We showed in \cite[Proposition~10.1]{douvropoulosjosuatverges} that the conjugacy class of $\underline{f}$ is invariant under this action.  This suggests to extend this action to $\CPF^{(m)}$, in a way that is compatible with the action of $W$: is there a natural action of $W \times \mathbb{Z}_{mh+2}$ on $\CPF^{(m)}$ ?  In type $A$, $B$, and  $I$, $\mathbb{Z}_{mh+2}$ acts by rotation of labeled dissections. 
\end{open}


\begin{open}
Show that $\CPF^{(m)}$ and $\CPF^{(m),+}$ are shellable complexes, and construct a shelling that gives a combinatorial proof and interpretation of Theorem~\ref{thm:h-vector}. 
\end{open}


\begin{open}
    Build an explicit basis $(\beta_\phi)_{\phi\in \PF^{(m)}}$ of $\tilde H_{n-1}(\CPF^{(m)})$ (respectively, an explicit basis $(\beta_\phi)_{\phi\in \PF'^{(m)}}$ of $\tilde H_{n-1}(\CPF^{(m),+})$), such that the action of $W$ is given by:
    \[
        w \cdot \beta_{\phi}
            =
        (-1)^{\ell(w)} \beta_{w\cdot \phi}.
    \]
    Indeed, the combinatorial nature of the group action on the homology suggests their existence.  (Note that we need to take the homology with rational coefficients to ensure the existence of such a basis.)
\end{open}





%

\begin{open}
    We defined cluster parking functions as a subposet of $\PF \times \Delta^{(m)}$.  In another direction, the parking function poset also has a Fuß analog $\PF^{(m)}$.  It has been described explicitly (in type $A$, but this extends immediately to the general case) in~\cite[Section~6]{delcroixogerjosuatvergesrandazzo}.  This suggests to mix both Fuß generalizations.  Let $m_1$ and $m_2$ be positive integers.  Can we define $m_1$-cluster $m_2$-parking functions as a subposet of $\PF^{(m_2)} \times \Delta^{(m_1)}$ and study their properties ?
\end{open}



\begin{open}
    Find an explicit isomorphism  of $W$-modules between $\tilde H_{n-1}(\CPF)$ and Gordon's ring $R_W$, defined in~\cite{gordon} as a quotient of diagonal coinvariants.
\end{open}




\begin{open}
    In type $A$, there are combinatorial definitions of both {\it rational} parking functions~\cite{armstrongloehrwarrington} and {\it rational} associahedra~\cite{armstrongrhoadeswilliams}.  Build on these to define rational cluster parking functions.
\end{open}

\end{document}